\documentclass[11pt,a4paper,makeidx, amstex]{amsart}
\usepackage[latin1]{inputenc}        
\usepackage[dvips]{graphics}

\usepackage
{graphicx}
\usepackage{epstopdf}
\usepackage{mathrsfs}
\let\mathcal\mathscr
\usepackage{color}		
\usepackage{epsfig}
\psdraft

\usepackage{amssymb}
\usepackage{amsmath}
\usepackage{amsfonts}
\usepackage{latexsym}
\usepackage[T1]{fontenc}
\usepackage[latin1]{inputenc}
\usepackage{hyperref}
\usepackage[all,ps]{xy}
\RequirePackage{amsthm}
\RequirePackage{amssymb}
\usepackage[all,ps]{xy}
\usepackage{latexsym}
\usepackage{amscd}
\usepackage[latin1]{inputenc}
\usepackage[T1]{fontenc}
\usepackage{color}
\RequirePackage{amsthm}
\RequirePackage{amssymb}
\usepackage[all,ps]{xy}
\usepackage{latexsym}

\def\Z{{\mathbb Z}}
\def\N{{\mathbb N}} 
\def\P{{\mathbb P}}

\def\Q{{\mathbb Q}}
\def\C{{\mathbb C}}

\def\Im{\mathop{\rm Im}\nolimits}

\def\lra{\longrightarrow}

\def\Pic{\mathop{\rm Pic}\nolimits}

\def\tilde{\widetilde}

\def\phi{\varphi}

\def\Pic{\mathop{\rm Pic}\nolimits}
\def\dim{\mathop{\rm dim}\nolimits}
\def\div{\mathop{\rm div}\nolimits}

\def\g{{\mathfrak g}}

\newcommand{\kntiposp}{K3$^{[n]}$-type }
\newcommand{\kntipo}{K3$^{[n]}$-type}
\newcommand{\reldim}{\mathrm{reldim}}

\newcommand{\ra}{\rightarrow}

\newtheorem{thm}{Theorem}[section]

\newtheorem{cor}[thm]{Corollary}
\newtheorem{definition}[thm]{Definition}
\newtheorem{rmk}[thm]{Remark}
\newtheorem{prop}[thm]{Proposition}
\newtheorem{lem}[thm]{Lemma}

\newtheorem{ex}[thm]{Example}

\newtheorem{ques}[thm]{Question}

\title[Rational curves and 0-cycles on holomorphic symplectic varieties]{Families of rational curves on holomorphic symplectic varieties and applications to 0-cycles}

\author{Fran\c{c}ois Charles}
\address{Laboratoire de Math\'ematiques d'Orsay, Universit\'e Paris-Sud, 91405 Orsay CEDEX, France}
\email{francois.charles@math.u-psud.fr}

\author{Giovanni Mongardi}
\address{Alma Mater studiorum Universit\'a di Bologna
Dipartimento di Matematica,
Piazza di Porta San Donato 5,
Bologna, 40126 Italia}
\email{giovanni.mongardi2@unibo.it}

\author{Gianluca Pacienza}
\address{Institut Elie Cartan de Lorraine,
Universit\'e de Lorraine,
B.P. 70239, F-54506 Vandoeuvre-l\'es-Nancy Cedex
 France}
\email{gianluca.pacienza@univ-lorraine.fr}

\date{\today}

\begin{document}

%
\begin{abstract}
We study families of rational curves on irreducible holomorphic symplectic varieties. 
We give a necessary and sufficient condition for a sufficiently ample linear system on a holomorphic symplectic variety  
of $K3^{[n]}$-type to contain a uniruled divisor covered by rational curves of primitive class. In particular, for any fixed $n$, we show that there are only finitely many polarization types of holomorphic symplectic variety  
of $K3^{[n]}$-type that do not contain such a uniruled divisor.
As an application we provide a generalization of a result due to Beauville-Voisin on the Chow group of 0-cycles on such varieties. 
\end{abstract}
%
%
%
%
\maketitle
{\let\thefootnote\relax
\footnote{
\noindent
\textbf{A.M.S.~classification :} 14C99, 14J28, 14J35, 14J40. \\
\noindent
FC is supported by the European Research Council (ERC) under the European Union's Horizon 2020 research and innovation programme (grant agreement No 715747).
GM is supported by ``Progetto di ricerca INdAM per giovani
ricercatori: Pursuit of IHS'' and would like to thank the Universit\'e de Lorraine for hosting him while this work was finished. GP was partially supported by the Projet ANR-16-CE40-0008  ``Foliage'', the 
 ANR project "CLASS'' no. ANR-10-JCJC-0111 and the University of Strasbourg Institute for Advanced Study (USIAS) as part of a USIAS Fellowship. 
 }}
\numberwithin{equation}{section}

\section{Introduction}

Let $S$ be a $K3$ surface and $H$  an ample divisor on $S$. By a theorem of Bogomolov and Mumford \cite{MoriMukai83}, the linear system $|H|$ contains an element whose irreducible components are rational.
A simple, yet  striking application of the existence of ample rational curves on any projective $K3$ surface $S$ has been given by Beauville and Voisin in \cite{BeauvilleVoisin04}. They remarked that any point on any rational curve on the $K3$ determines the same (canonical) zero-cycle $c_S$ of degree $1$, and proved that the image of the intersection product 
$\Pic(S)\otimes \Pic(S)\to CH_0(S)$ is contained in $\Z \cdot c_S$. Tensoring with $\mathbb Q$ we may restate these two results as follows as an equality between the following three groups:
\begin{eqnarray*}
 \Im( (j_1)_*:CH_0(R_1)_\Q \to CH_0(S)_\Q) &=& \Im( (j_2)_*:CH_0(R_2)_\Q\to CH_0(S)_\Q)\\&=&
 \Im(\Pic(S)_\Q \otimes \Pic(S)_\Q\to CH_0(S)_\Q),
\end{eqnarray*}
where $j_i:R_i\hookrightarrow S,\ i=1,2,$ are any two rational curves on the $K3$ surface $S$. 

The goal of this paper is to investigate the extent to which Bogomolov-Mumford's and  Beauville-Voisin's results can be generalized to the higher-dimensional setting. 

Let $X$ be a compact K\"ahler manifold. We say that $X$ is irreducible holomorphic symplectic -- in the text, we will often simply refer to such manifolds as holomorphic symplectic -- if $X$ is simply connected and $H^0(X, \Omega^2_X)$ is spanned by an everywhere non-degenerate form of degree $2$. These objects were introduced by Beauville in \cite{Beauville83}. The holomorphic symplectic surfaces are the $K3$ surfaces.

The cohomology group $H^2(X, \Z)$ is endowed with a natural quadratic form $q$, the \emph{Beauville-Bogomolov form}. We denote it by $q$.

It is to be noted that holomorphic symplectic varieties with $b_2>3$ are not hyperbolic. 
This has been proved by Verbitsky, cf. \cite{Verbitsky13, verbitsky2017ergodic}, using among other things his global Torelli theorem \cite{Verbitsky09}. Much less seems to be known on the existence of rational curves on (projective) holomorphic symplectic varieties.

In order to investigate those, we make the following definition.
\begin{definition}
Let $C$ be a stable curve of genus $0$, and let $f : C\ra X$ be a morphism that is unramified at all the generic points of $C$. We say that the curve $f(C)$ in $X$ is \emph{ruling} if there exists a family 
$$p : \mathcal{C}\ra S$$
of stable curves over an irreducible, quasi-projective base $S$, a point $0\in S$, and a a morphism 
$$\phi : \mathcal{C}\lra X$$
such that $C=p^{-1}(0)$, $\phi_{|C}=f,$ and $p(C)$ has codimension $1$ in $X$. 

We say that $p(\mathcal C)$ is a \emph{uniruled divisor}, that it is ruled by the stable curve $f : C\ra X$ -- or, for short, ruled by $f(C)$.
\end{definition}

A stable genus $0$ curve in $X$ is by definition a morphism $f : C\ra X$ as above.

Given a curve $R$ in $X$, we may consider the  class $[R]$ of $R$ in $H_2(X, \Q)$. Let $[R]^\vee\in H^2(X,\Q)$ be the Poincar\'e dual of $[R]$. Then $[R]^\vee$ is the class of a divisor in $X$. We say that $R$ is \emph{positive} if $q([R]^\vee)>0$, and that $R$ is \emph{ample} if $[R]^\vee$ is an ample class.


Recall that a holomorphic symplectic manifold is said to be of $K3^{[n]}$-type if it is a deformation of the Hilbert scheme that parametrizes zero-dimensional subschemes of length $n$ on some $K3$ surface. 
Our main result is the following -- we will prove it in a slightly more precise and technical form in Theorem \ref{thm:precise} and Proposition \ref{prop:curvegrandi} below.

\begin{thm}\label{thm:main}
Let $n\geq 1$ be an integer. Let $\mathfrak M=\cup_{d>0}\mathfrak M_{2d}$  be the union
of the moduli spaces $\mathfrak M_{2d}$ of projective irreducible holomorphic symplectic varieties of $K3^{[n]}$-type polarized by a 
line bundle of degree $2d$.  For all $(X,H)\in \mathfrak M$, outside at most a finite number of connected components, 
the following holds:
\begin{enumerate}
\item There exists a ruling genus $0$ stable curve in $X$ with cohomology class proportional to the Poincar\'e-dual of the class of $H$;
\item there exists a positive integer $m$ such that the linear system $|mH|$ contains a uniruled divisor.
\end{enumerate}

\end{thm}

\begin{rmk}
{\rm {The ruling curve in (1) may be chosen to have primitive cohomology class, see the comments below. However, we are not able to control the integer $m$ in (2).}}
\end{rmk}

Some comments are in order. The Beauville-Bogomolov form induces an embedding 
$H^2(X,\mathbb Z)\hookrightarrow H_2(X,\mathbb Z), H\mapsto H^\vee$. 
By abuse of notation we denote again by $q$ the quadratic form on $H_2(X,\mathbb Z)$. We can make explicit the components of $\mathfrak M$ for which the existence is obtained:
\begin{rmk}\label{rmk:precisestate}
{\rm {
 The statement above splits into two parts. On the one hand Theorem \ref{thm:precise} ensures the existence of uniruled divisors covered by primitive rational curves if there exist 
integers $p,g$ and $\epsilon$ such that $p\geq g$ and $\epsilon =0$ or $1$ such that the following two conditions hold:
\begin{enumerate}
\item[(i)] the class $\alpha:=\frac{H^\vee}{\div(H)}\in H_2(X,\mathbb Z)$ can be written as $\gamma +(2g-\epsilon)\eta$, with 
$\eta$ in the monodromy orbit of the class of the exceptional curve on a $K3^{[n]}$ and $\gamma\in \eta^\perp$;
\item[(ii)] $q(\gamma)=2p-2$ (hence $q(\alpha)=2p-2-\frac{(2g-\epsilon)^2}{2n-2}$).
\end{enumerate}
On the other hand, thanks to Proposition \ref{prop:curvegrandi}, we can show that the two conditions above are satisfied outside at most a finite number of connected components.
Observe furthermore that conditions (i) and (ii) above determine the monodromy orbit of the polarization $H$ (cf. Corollary \ref{cor:control-of-polarization-dual} for details).}}
\end{rmk}
We list below some relevant cases in which the conditions (i) and (ii) of Remark \ref{rmk:precisestate} are easily seen to be satisfied. 

\begin{rmk}\label{rmk:sufficient}
{\rm {
\begin{enumerate}
\item[(i)] If $q(\alpha)\geq n-1$, 
then a multiple of $H$ is uniruled by primitive rational curves of class $\alpha$ (see Proposition \ref{prop:curvegrandi}). 

\item[(ii)] If $\rho(X)\geq 2$ then $X$ always contains an ample uniruled divisor covered  by primitive rational curves (cf. Corollary \ref{cor:rho2}). 
\item[(iii)] If $n\leq 7$ then the conclusion of the theorem holds for {\it all} the connected components of $\mathfrak M$ (cf. Remark \ref{rmk:npiccolo}).
\item[(iv)] If $n-1$ is a power of a prime number, then by \cite[Lemma 9.2 and subsequent comment]{Markman11}, the monodromy group is maximal. Therefore it suffices to check that the square $q(\alpha)$
is of the form  $2p-2-\frac{(2g-\epsilon)^2}{2n-2}$, with $p\geq g$. 
\end{enumerate}
}}

\end{rmk}

%
%
%

The existence of uniruled divisors ruled by primitive rational curves on {\it any} projective holomorphic symplectic variety of $K3^{[n]}$-type 
was wrongly claimed in \cite{CP}. Counterexamples were recently provided by Oberdieck, Shen  and  Yin in \cite[Corollary A.3]{zal}.
The proof presented in \cite{CP} was based upon the following three ingredients:
(a) the existence of a ``controlled'' polarized deformation of a polarized holomorphic symplectic variety $(X,H)$
 to a $((K3)^{[n]}, H')$, as a consequence of Verbitsky's  global Torelli theorem and Markman's study of the monodromy group (see Section 2); (b) a geometric criterion to deform a rational curve on a holomorphic symplectic variety $X$ along its Hodge locus inside the moduli space of $X$, which we derive from Mumford's theorem on 0-cycles and deformation theoretic arguments (cf. Section 3); (c)  the existence of uniruled divisors on a $(K3)^{[n]}$, via points on nodal curves in the hyperplane linear system (see Section 4). These three parts of the proof are correct and will be presented here essentially as in \cite{CP}. The main difference is that, after the appearance of \cite{zal} we realised that the examples we provided in (c) did not (and actually could 
 not, because of \cite[Corollary A.3]{zal}) cover all the connected components of $\mathfrak M$. The same type of considerations and results hold in the generalized Kummer case, treated in \cite{MP} and amended in \cite{MPcorr}.

In the present paper we also show that conditions (i) and (ii) in Remark \ref{rmk:precisestate} are precisely satisfied in {\it all the cases} where the obstruction discovered by \cite{zal} does not prevent
uniruled divisors covered by primitive rational curves to exist. In other words
we show that our result is sharp (see subsection \ref{ss:k3no} for all the details). 
In all the cases where the theorem fails, using \cite{klm2}, we can check the existence of 
a family of the expected dimension of rational curves covering a coisotropic subvariety of codimension $c\geq 2$ (cf. Proposition \ref{prop:ultima}. Moreover in some of these cases we can actually prove that 
$c=2$ (see Theorem \ref{thm:codim2} in subsection \ref{ss:codim2}). 

In Subsection 5.2 we discuss  to which extent it might be possible to find uniruled divisors covered by possibly non-primite rational curves. We find an explicit condition (cf. Prop. \ref{prop:mbound}) for the existence of such divisors
 which can be easily checked on  examples. For instance this allows us to check 
the existence of ample uniruled divisors covered by non-primitive rational curves {\it for all the connected components} of the moduli 
space $\mathfrak M$ up to dimension $26$. 
 
 It should be noted that applications of the existence of uniruled divisors to the study of Chow groups of $0$-cycles does not make use of the primitivity of the relevant rational curves.
These applications are presented in Section 6. 
To state our results recall that if $Y$ is a variety,  $CH_0(Y)_{hom}$ is the subgroup of $CH_0(Y)$ consisting of zero-cycles of degree zero.
\begin{definition}\label{definition:S1}
Let $D$ be an irreducible divisor on $X$. We denote by $S_1 CH_0(X)_{D,hom}$ the subgroup  $$S_1 CH_0(X)_{D,hom}:=\Im \big( CH_0(D)_{hom}\to CH_0(X)\big),$$
of $CH_0(X)$. 
We denote by $S_1 CH_0(X)_{D}$ the subgroup 
$$S_1 CH_0(X)_{D}:=\Im \big( CH_0(D)\to CH_0(X)\big)$$
of $CH_0(X)$.
\end{definition}
We prove the following. 


\begin{thm}\label{thm:appliA}
Let $X$ be a projective holomorphic symplectic variety.  Suppose that $X$ 
possesses an ample ruling curve. Then  the subgroups $S_1 CH_0(X)_{D,hom}$ and $S_1 CH_0(X)_{D}$ are independent of the irreducible uniruled divisor $D$.
\end{thm}


In light of the result above we set 
$$S_1 CH_0(X)_\Q:= \Im( j_*:CH_0(D)_\Q\to CH_0(X)_\Q)$$
where $j:D\hookrightarrow X$ is any irreducible uniruled divisor. It is natural to ask whether such a subgroup of $CH_0(X)_\Q$ has an intersection-theoretic interpretation, as for  $K3$'s. This is indeed the case. 


\begin{thm}\label{thm:appliB}
Let $X$ be a projective holomorphic symplectic variety. Suppose that $X$ 
possesses an ample ruling curve and that the group of cohomology classes of curves on $X$ is generated over $\Q$ by classes of ruling curves. Then for any nontrivial $L\in \Pic(X)$ we have
$$
S_1 CH_0(X)_{hom}= L\cdot CH_1(X)_{hom} \ \ {\textrm{and }}\ 
 S_1 CH_0(X)= L\cdot CH_1(X).$$
\end{thm}

In particular the conclusions of the results above hold for all projective holomorphic symplectic variety of $K3^{[n]}$-type, for $n\leq 13$, and, for higher $n$, for all but finitely many components of $\mathfrak M$, as specified in Remark \ref{rmk:precisestate}. The hypothesis on the Picard group is verified in applications by showing the existence of uniruled divisors linearly equivalent to certain multiple of each ample divisor. There is no evidence this could not hold in general.

The theorems above may be regarded as a higher dimensional analogue of the Beauville-Voisin result. 
It is important to notice that, 
for holomorphic symplectic varieties of higher dimension, Beauville has stated in \cite{Beauville07} a far-reaching conjectural generalization of their result, called ``weak splitting property'', which can be deduced by a  (conjectural) splitting of the (conjectural) Bloch-Beilinson filtration on the Chow group. 
The conjecture was further refined by Voisin in  \cite{Voisin08}. 
These conjectures have been intensively studied in the last years. 
After 
the appearence of a first version of this paper Voisin \cite{Voi15} has unveiled a surprising conjectural connection
between the weak splitting property conjecture and the existence of subvarieties whose zero cycles are supported in lower dimension. A first instance of such an existence result is provided by Theorem \ref{thm:main}. For a (non-exhaustive) list of works in this research direction see \cite{Voisin08, Fe12a, Fu13a, Fu13b, huybrechts2013curves, Riess14, voisin2015rational, lat-rem, lin-lagr, lin-kummer, LP2, yinfinite, SV, MP, vial2017motive, shen2017derived, shen2017k3, FLVS, zal, voisin2018triangle}.

%

We end the paper with some speculations on possible generalizations of our results.

After the appearance of \cite{CP}, other articles studied the existence of rational curves on projective holomorphic symplectic varieties. We already mentioned \cite{MP, MPcorr} where the analogous existence results of ample uniruled divisors are established for deformations of generalized Kummer varieties. The existence of primitive rational curves moving in a family of the expected dimension on projective holomorphic symplectic varieties deformations of punctual Hilbert schemes on a $K3$ surface or of generalized Kummer varieties is shown in \cite[Theorem 3.2]{MO} and \cite[Theorem 5.1]{MP}.  
While finishing the first version of the paper \cite{CP}, independent work of Amerik and Verbitsky \cite{AmerikVerbitsky14} has appeared and Section 4 of \cite{AmerikVerbitsky14} have some overlaps with results and arguments presented here in section 3 concerning the deformations of rational curves on holomorphic symplectic varieties. Amerik and Verbitsky are more concerned with negative rational curves, while we focus on positive ones (i.e. dual to an ample class). Since the goals and the results of the two papers are quite different, for the sake of completeness we did not try to eliminate similar discussions. We refer the reader to Section 3 below for the precise references to the similar results appearing in \cite{AmerikVerbitsky14}.

\noindent{\bf Acknowledgements.} We thank Daniel Huybrechts and Claire Voisin for interesting discussions on the subject of this paper. In particular C. Voisin kindly suggested the applications to Chow groups presented here. We are  grateful to Eyal Markman for enlightening discussions and comments, and for sharing with us some unpublished notes of his. 
We are grateful to Apostol Apostolov for his help with Theorem \ref{thm:control-of-polarization}. We thank Qizheng Yin for valuable discussions on \cite{zal}. 
 Thanks to Giuseppe Ancona for a useful suggestion. We thank Giulia Sacc\`a for pointing out to us that the proof of former Corollary 3.6 was incomplete, which led us to modify accordingly the manuscript. 
\bigskip

We always work over the field $\C$ of complex numbers.

%

\section{Varieties of $K3^{[n]}$-type and their polarizations}

In this section we collect a number of known results on Hilbert schemes of points on a $K3$ surface and then move on to study the {\it polarized} deformations of  irreducible holomorphic symplectic varieties of $K3^{[n]}$-type.

\subsection{Lattices.}
The general theory of irreducible holomorphic symplectic varieties as in \cite{Beauville83} shows that the group $H^2(X, \Z)$ is endowed with a natural symmetric bilinear form $q_X$ of signature $(3, 20)$, the \emph{Beauville-Bogomolov form}. When no confusion is possible we will denote the square $q_X(h)$ of an element $h\in H^2(X, \Z)$ simply by $q(h)$ or  by $h^2$. Therefore $(H^2(X, \Z), q_X)$ is naturally a lattice.  We try to use the notation $h$ for elements of the cohomology lattice 
and the capital letter $H$ for divisors/line bundles (but no confusion should hopefully arise if we do otherwise somewhere!).

\begin{definition}\label{def:div}
Let $\Lambda$ be a free $\Z$-module of finite rank endowed with a symmetric bilinear form. If $h$ is an element of $\Lambda$, the \emph{divisibility} of $h$ is the nonnegative integer $t$ such that 
$$h\cdot \Lambda=t\Z.$$ It will be denoted by $\div(h)$.
If $(X, h)$ is a polarized irreducible holomorphic symplectic variety, then the divisibility of $h$ is its divisibility as an element of the lattice $H^2(X, \Z)$ endowed with the Beauville-Bogomolov form.
\end{definition}

We will use the following maps relating a lattice $(\Lambda,q)$ and its dual $\Lambda^\vee$. 

There is a natural map
\begin{equation}\label{eq:emb}
\Lambda \hookrightarrow \Lambda^\vee,\ \ \ \lambda\mapsto q(\lambda,\cdot).
\end{equation}

Given a class $c \in \Lambda^\vee$ there exists a unique class 
$\lambda_c \in \Lambda \otimes \Q$ such that for all $\lambda'\in \Lambda$ we have
$$
 q(\lambda_c,\lambda')=c(\lambda'). 
$$
This induces a natural map
\begin{equation}\label{eq:embdual}
\Lambda^\vee \hookrightarrow \Lambda\otimes \mathbb Q, \text{     }
c\mapsto \lambda_c.
\end{equation}
Using the above we can endow $\Lambda^\vee$ with a quadratic form taking {\it rational} values. By abuse of notation 
we will still denote it by $q$. 

We will often make use of the following set-theoretic map:
\begin{equation}\label{eq:nonlin}
\Lambda \hookrightarrow \Lambda\otimes \Q,\ \ \ \lambda\mapsto \frac{\lambda}{\div(\lambda)}.
\end{equation}
One easily checks that the image is contained in the image of $\Lambda^\vee$ under the map (\ref{eq:embdual}) and gives all primitive
elements in it.  

Finally the above map will be composed with the projection onto the {\it discriminant group} $\Lambda^\vee/\Lambda$:
\begin{equation}\label{eq:nonlinclass}
\Lambda \hookrightarrow \Lambda^\vee \to \Lambda^\vee/\Lambda,\ \ \ \lambda\mapsto \Big[\frac{\lambda}{\div(\lambda)}\Big].
\end{equation}

\subsection{Hilbert schemes of points on a $K3$ surface.}
Let $S$ be a compact complex projective surface. If $n$ is a positive integer, denote by $S^{[n]}$ the Hilbert scheme (or the Douady space in the non-projective case) of length $n$ subschemes of $S$. By \cite{Fogarty68}, $S^{[n]}$ is a smooth complex variety. The general theory of the Hilbert scheme shows that $S^{[n]}$ is projective if $S$ is.

Assume that $S$ is a $K3$ surface. By \cite{Beauville83}, $S^{[n]}$ is an irreducible holomorphic symplectic variety: it is simply connected, and the space $H^0(S^{[n]}, \Omega^2_{S^{[n]}})$ is one-dimensional, generated by a holomorphic symplectic form. Let $X$ be an irreducible holomorphic symplectic variety. We say that $X$ is of $K3^{[n]}$-type if $X$ is deformation equivalent, as a complex variety, to $S^{[n]}$, where $S$ is a complex $K3$ surface. 

Let $S$ be a $K3$ surface, and let $n>1$ be an integer. We briefly recall the description of the group $H^2(S^{[n]}, \Z)$ as in \cite[Proposition 6]{Beauville83}. 

Let $S^{(n)}$ be the $n$-th symmetric product of $S$, and let $\epsilon : S^{[n]}\ra S^{(n)}$ be the Hilbert-Chow morphism. The map 
$$\epsilon^* : H^2(S^{(n)}, \Z)\ra H^2(S^{[n]}, \Z)$$
is injective. Furthermore, let $\pi : S^n\ra S^{(n)}$ be the canonical quotient map, and let $p_1, \ldots, p_n$ be the projections from $S^n$ to $S$. There exists a unique map 
$$i : H^2(S, \Z)\ra H^2(S^{[n]}, \Z)$$
such that for any $\alpha\in H^2(S, \Z)$, $i(\alpha)=\epsilon^*(\beta)$, where $\pi^*(\beta)=p_1^*(\alpha)+\ldots+p_n^*(\alpha)$. The map $i$ is an injection.

Let $E_n$ be the exceptional divisor in $S^{[n]}$, that is, the divisor that parametrizes non-reduced subschemes (we will drop the index $n$ and simply write $E=E_n$ when no confusion is possible). The cohomology class of $E$ in $H^2(S^{[n]}, \Z)$ is uniquely divisible by $2$ -- see \cite[Remarque after Proposition 6]{Beauville83}. Let $\delta:=\delta_n$ be the element of $H^2(S^{[n]}, \Z)$ such that $2\delta_n=[E]$. Then we have
\begin{equation}\label{H2}
H^2(S^{[n]}, \Z)=H^2(S, \Z)\oplus_{\perp}\Z\delta_n,
\end{equation}
where the embedding of $H^2(S, \Z)$ into $H^2(S^{[n]}, \Z)$ is the one given by the map $i$ above.

The decomposition (\ref{H2}) is orthogonal with respect to the Beauville-Bogomolov form $q$, and the restriction of $q$ to $H^2(S, \Z)$ is the canonical quadratic form on the second cohomology group of a surface induced by cup-product. We have 
$$\delta^2=-2(n-1).$$

 By Poincar\'e duality, $H_2(S^{[n]}, \Z)$ may be identified to the dual lattice of $H^2(S^{[n]}, \Z)$.
Therefore, using (\ref{eq:embdual}),
to a class $Z\in H_2(S^{[n]}, \Z)$ we can associate a unique class 
$D_Z\in H^2(S^{[n]}, \Q)$ such that for all $D'\in H^2(S^{[n]}, \Z)$ we have
$$
 q(D_Z,D')=Z\cdot D'. 
$$
The class $D_Z$ will be called the {\it dual} of $Z$ with respect to the Beauville-Bogomolov quadratic form. 
In this way we obtain a quadratic form on the homology group $H_2(S^{[n]}, \Z)$ taking rational values and such that 
$$
H_2(S^{[n]}, \Z)= H_2(S, \Z)\oplus^{\perp} \Z r_n
$$
where $r_n$ is the homology class orthogonal to $H_2(S, \Z)$ and such that 
\begin{equation}\label{eq:r.delta}
r_n\cdot \delta_n=-1.
\end{equation}
In particular we have 
$$
 r_n = \frac{1}{2(n-1)} \delta_n
$$
which implies
$$
 q(r_n) = -\frac{1}{2(n-1)}.
$$
 Geometrically $r_n$ is the class of an exceptional rational curve which is the general fiber of the Hilbert-Chow morphism (see e.g. \cite{HassettTschinkel10}). 
 
 By abuse of notation, if $h\in H^2(S,\Z)$, we will again denote by $h$ the induced class
 in $H^2(S^{[n]}, \Z)$ as well as that in $H_2(S^{[n]}, \Z)$, using the embedding (\ref{eq:emb}).

\subsection{Polarized deformations of varieties of $K3^{[n]}$-type}
In this paper, we are interested in the possible deformation types for primitively polarized varieties $(X, h)$, where $X$ is a variety of $K3^{[n]}$-type and $h$ is a primitive polarization of $X$, that is, the numerical equivalence class of a primitive and ample line bundle on $X$. 

\begin{definition}
Let $X$ and $X'$ be two compact complex manifolds, and let $h, h'$ be numerical equivalence classes of line bundles on $X$ and $X'$ respectively. We say that the pairs $(X, h)$ and $(X', h')$ are deformation equivalent if there exists a connected complex variety $S$, a smooth, proper morphism $\pi : \mathcal X\ra S$, a line bundle $\mathcal L$ on $\mathcal X$ and two points $s, s'$ of $S$ such $(\mathcal X_s, c_1(\mathcal L_s))$ is isomorphic to $(X, h)$ and $(\mathcal X_{s'}, c_1(\mathcal L_{s'}))$ is isomorphic to $(X', h')$.
\end{definition}

\begin{rmk}\label{remark:projective-def-classes}
{\em Let $X$ be a variety of $K3^{[n]}$-type, and let $h$ be a polarization on $X$. Then Markman shows in \cite[Proposition 7.1]{Markman11} that there exists a $K3$ surface $S$ and a polarization $h'$ on $S^{[n]}$ such that $(X, h)$ is deformation equivalent to $(S^{[n]}, h')$. 
}
\end{rmk}

\bigskip

In the surface case, that is, when $n=1$, the global Torelli theorem implies that two primitively polarized $K3$ surfaces $(X, h)$ and $(X', h')$  are deformation equivalent if and only $h^2=h'^2$. 
The situation is different in higher dimension. 

%
%
Let $X$ be a variety of $K3^{[n]}$-type. If $n=1$, the lattice $H^2(X, \Z)$ is unimodular, so that the divisibility of any nonzero primitive element of $H^2(X, \Z)$ is $1$. This is no longer the case as soon as $n>1$.

By (\ref{eq:nonlin}) we have the following map of sets 
\begin{equation}\label{eq:dualset}
 H^2(X, \Z) \to  H^2(X, \Q),\ \ h\mapsto \frac{1}{\div(h)} h.
\end{equation}

Let $(X, h)$ be a primitively polarized irreducible holomorphic symplectic variety. Both $h^2$ and the divisibility of $h$ are constant along deformations of $(X, h)$. However, as shown in \cite{Apo}, it is not true that these two invariants determine the deformation type of $(X, h)$. In this section, we will give explicit representatives for all the deformation-equivalence classes of primitively polarized varieties of $K3^{[n]}$-type. 

\bigskip

We start by describing results due to Markman on deformation-equivalence of polarized varieties of $K3^{[n]}$-type. These results rely both on the global Torelli theorem \cite{Verbitsky09} and the computation of the monodromy group of varieties of $K3^{[n]}$-type in \cite{Markman10}. Recently, Kreck and Su \cite{kreck} provided a counterexample to some
of the statements contained in Verbitsky's global Torelli theorem.
However, this does not affect the results we are using, as they rely
on Markman's formulation of the global Torelli theorem using marked
moduli spaces instead of the Teichmuller space used by Verbitsky. We
refer the readers to \cite[Theorem 3.1 and Remark 3.3]{Looijenga} for the correct
statement and a comment about the difference between the Teichmuller
space and marked moduli spaces with respect to the global Torelli
theorem (see also \cite{Verbitsky-erratum}).

Let $S$ be a $K3$ surface, and $n>1$ an integer. Let $\widetilde\Lambda$ be the Mukai lattice of $S$
$$\widetilde \Lambda=H^0(S, \Z)\oplus H^2(S, \Z)\oplus H^4(S, \Z)$$
endowed with the quadratic form defined by 
$$\langle (a, b, c), (a', b', c')\rangle=bb'-ac'-a'c.$$
Let $v_n=(1, 0, 1-n).$ We identify $H^2(S^{[n]}, \Z)$ endowed with the Beauville-Bogomolov quadratic form with the orthogonal of $v_n$ in $\widetilde \Lambda$. The inclusion 
$$H^2(S, \Z)\hookrightarrow (v_n)^{\perp}$$
is compatible with the decomposition (\ref{H2}).

If $h$ is any class in $H^2(S^{[n]}, \Z)\subset \widetilde\Lambda$, let $T_S(h)$ be the saturation in $\widetilde\Lambda$ of the lattice spanned by $h$ and $v_n$. 

\begin{prop}\label{proposition:deformation-criterion}
Let $S$ and $S'$ be two $K3$ surfaces, and let $n>1$ be an integer. Let $h$ (resp. $h'$) be the numerical equivalence class of a big line bundle on $S^{[n]}$ (resp. $S'^{[n]}$). The pairs $(S^{[n]}, h)$ and $(S'^{[n]}, h')$ are deformation equivalent if and only if there exists an isometry
$$T_{S}(h)\ra T_{{S'}}(h')$$
mapping $h$ to $h'$.
\end{prop}

\begin{proof}
In case $h$ and $h'$ are ample, this is the result of Markman written up in \cite[Proposition 1.6]{Apo}.

In the general case, choose small deformations $(X, h)$ and $(X', h')$ of $(S^{[n]}, h)$ and $(S'^{[n]}, h')$ respectively such that both $X$ and $X'$ have Picard number $1$. By a theorem of Huybrechts \cite{Huybrechts99}, $h$ and $h'$ are ample classes on $X$ and $X'$ respectively. 

The construction of the rank $2$ lattices $T_S(h)$ and $T_{S'}(h')$ generalizes to $X$ and $X'$ to provide rank $2$ lattices $T_X(h)$ and $T_{X'}(h')$. This follows from the work of Markman as in \cite{Markman11}, Corollary 9.5. We refer to \cite{Markman11} and the discussion in \cite{Apo}, section 1 for the precise construction. 

The formation of $T_X(h)$ is compatible with parallel transport. As a consequence, the isomorphism $T_S(h)\ra T_{S'}(h')$ mapping $h$ to $h'$ induces an isomorphism $T_X(h)\ra T_{X'}(h')$ mapping $h$ to $h'$. It follows once again from \cite[Proposition 1.6]{Apo}, that $(X, h)$ and $(X', h')$ are deformation equivalent.
\end{proof}

We now state the main result of this section.

\begin{thm}\label{thm:control-of-polarization}
Let $n>1$ be an integer and let $(X, h)$ be a primitively polarized irreducible holomorphic symplectic variety of $K3^{[n]}$-type. Let $t$ be the divisibility of $h$, and let $I\subset \Z$ be a system of representatives of $\Z/t\Z$, up to the action of $-1$ on $\Z/t\Z$. Then $t$ divides $2n-2$ and there exists a $K3$ surface $S$, a primitive polarization $h_S$ on $S$ and an integer $\mu\in I$ such that the pair $(X, h)$ is deformation equivalent to $(S^{[n]}, th_S-\mu\delta_n)$.
\end{thm}

\begin{rmk}
{\em The class $th_S-\mu\delta_n$ is not ample in general. However, the argument of the proof of Proposition \ref{proposition:deformation-criterion} shows that it is big.
}
\end{rmk}


\begin{proof}
Remark \ref{remark:projective-def-classes} allows us to assume that $X=S^{[n]}$ for some $K3$ surface $S$. Let $\widetilde\Lambda=H^0(S, \Z)\oplus H^2(S, \Z)\oplus H^4(S, \Z)$ be the Mukai lattice of $S$ and let $v_n=(1, 0, 1-n).$ 

Write $2d$ and $t$ for the Beauville-Bogomolov square and the divisibility of $h$ respectively. We start by describing the structure of the lattice $T_{S}(h)$. 

Write 
$$h=(\mu, \lambda h_S, \mu (n-1))=\lambda h_S -\mu\delta_n$$
where $\lambda$ and $\mu$ are two integers and $h_S$ is primitive. Since $h$ is primitive, $\lambda$ and $\mu$ are relatively prime. It is readily checked that the divisibility of $h$ is 
$$t=\gcd(\lambda, 2n-2)$$
and that $\mu$ and $t$ are relatively prime.

The element 
\begin{equation}\label{generator-overlattice}
w=\frac{1}{t}h-\frac{\mu}{t}v_n=\Big(0, \frac{\lambda}{t}h_S, \frac{\mu(2n-2)}{t}\Big)\in \widetilde\Lambda
\end{equation}
belongs to $T_{S}(h)$, and the computation of \cite[Proposition 2.2]{Apo} show that $w$ generates the group $T_{S}(h)/(\Z h\oplus\Z v_n)\simeq \Z/t\Z.$

The lattice $N$ spanned by $h$ and $v_n$ in $\widetilde\Lambda$ is isomorphic to $\langle 2d\rangle\oplus\langle 2n-2\rangle$. Its discriminant group 
$N^\vee/N$ is $\Z/2d\Z\oplus \Z/(2n-2)\Z$. 

As in paragraph 4 of \cite{Nikulin80} and under the identifications above, the inclusion 
$$\langle 2d\rangle\oplus\langle 2n-2\rangle\simeq \Z h\oplus\Z v_n\subset T_{S}(h)$$
induces an injective morphism
$$\phi : \Z/t\Z\simeq T_{S}(h)/(\Z h\oplus\Z v_n)\hookrightarrow \Z/2d\Z\oplus \Z/(2n-2)\Z.$$
By (\ref{generator-overlattice}), $\phi$ sends $1$ to $(2d/t, \mu(2n-2)/t)$.

\bigskip

Let $S'$ be a $K3$ surface with a primitive polarization $h_{S'}$. Let $\mu'$ be an arbitrary integer, and define 
$$h'=th_{S'}-\mu'\delta_n\in H^2(S'^{[n]}, \Z)=H^2(S', \Z)\oplus\Z\delta_n.$$
By \cite[Proposition 1.4.2]{Nikulin80} and the discussion above, a sufficient condition ensuring that there exists an isomorphism of lattices 
$T_{S}(h)\ra T_{S'}(h')$ sending $h$ to $h'$ is that 
$$t^2h_{S'}^2-\mu'^2(2n-2)=2d$$
and 
$$\frac{\mu'(2n-2)}{t}=\pm\frac{\mu(2n-2)}{t} \mod 2n-2,$$
i.e.
$$\mu'=\pm\mu \mod t.$$

Now let $\mu'$ be an integer in $I$ such that 
$$\mu'=\pm \mu \mod t.$$
Then $\mu'$ is prime to $t$ since $\mu$ is. Furthermore, $\lambda$ is divisible by $t$ and $h_S^2$ is even, so that $\lambda^2h_S^2$ is divisible by $2t^2$. Also, $(\mu^2-\mu'^2)(2n-2)$ is divisible by $2t^2$: both $(\mu^2-\mu'^2)$ and $2n-2$ are divisible by $t$, and at least one of these terms is divisible by $2t$, depending on the parity of $t$. As a consequence, the integer 
$$2d+\mu'^2(2n-2)=\lambda^2h_S^2-(\mu^2-\mu'^2)(2n-2)$$ 
is divisible by $2t^2$. Write 
$$2d=t^2.2d'-\mu'^2(2n-2),$$
and let $S'$ be a $K3$ surface with a primitive polarization $h_{S'}$ of degree $2d'$. 

By construction, there exists an isomorphism of lattices $T_{S}(h)\ra T_{S'}(h')$ sending $h$ to $h'$. Proposition \ref{proposition:deformation-criterion} shows that $(S^{[n]}, h)$ and $(S'^{[n]}, th_{S'}-\mu'\delta_n)$ are deformation equivalent.
\end{proof}

The following is an immediate consequence of the theorem.

\begin{cor}\label{cor:control-of-polarization-weak}
Let $n>1$ be an integer and let $(X, h)$ be a primitively polarized irreducible holomorphic symplectic variety of $K3^{[n]}$-type. Let $I\subset\Z$ be a system of representatives of $\Z/(2n-2)\Z$, up to the action of $-1$ on $\Z/(2n-2)\Z$. Then there exists a positive integer $m$, a $K3$ surface $S$, a primitive polarization $h_S$ on $S$ and an integer $\mu\in I$ such that the pair $(X, mh)$ is deformation equivalent to $(S^{[n]}, (2n-2)h_S-\mu\delta_n)$.
\end{cor}
We will use the dual statement on curve classes. 

\begin{cor}\label{cor:control-of-polarization-dual}
Let $n>1$ be an integer and let $X$ be a primitively polarized irreducible holomorphic symplectic variety of $K3^{[n]}$-type. Let $C\in H_2(X,\Z)$ be a primitive curve class with positive square. 
 Then there exists a $K3$ surface $S$, a primitive polarization $h_S$ on $S$ and an integer $\mu\in [0, n-1)$ such that there exists a polarized deformation
  from $X$ to $S^{[n]}$, carrying $C$ to $h_S-\mu r_n$.
\end{cor}
\begin{proof}
Following (\ref{eq:embdual}) and (\ref{eq:nonlin})
we write $C=h/div(h)$, for some primitive and positive $h\in H^2(X,\Z)$. As in the proof of Proposition \ref{proposition:deformation-criterion} we may assume that 
$h$ is ample. 
By Corollary \ref{cor:control-of-polarization-weak} there exists a polarized $K3$ surface $(S,h_S)$ and an integer 
$\mu\in [0, n-1)$ such that the pair $(X, mh)$ is deformation equivalent to $(S^{[n]}, (2n-2)h_S-\mu\delta_n)$, for  some integer $m>0$.
 The divisibility of $(2n-2)h_S-\mu\delta_n$ equals $2n-2$. 
Therefore, via the map (\ref{eq:dualset}), we get 
$$
(2n-2)h_S-\mu\delta_n \mapsto \frac{1}{2n-2} \big((2n-2)h_S-\mu\delta_n\big) = h_S - \mu r_n.
$$

This class yields the desired parallel transport of $C$.
\end{proof}

\begin{rmk}\label{rmk:traslo}
{\rm{The proofs of the two corollaries above work for any choice of a system of representatives of $\Z/(2n-2)\Z$, up to the action of $-1$ on $\Z/(2n-2)\Z$.
}}
\end{rmk}

\section{Deforming rational curves}

Let $\pi : \mathcal X\ra B$ be a smooth projective morphism of complex quasi-projective varieties of relative dimension $d$, and let $\alpha$ be a global section of Hodge type $(d-1,d-1)$ of the local system $R^{2d-2}\pi_*\Z$. Fixing such a section $\alpha$, we can consider the relative Kontsevich moduli stack of genus zero stable curves $\overline{{\mathcal M_0}}(\mathcal X/B, \alpha)$. We refer to \cite{BehrendManin96, FultonPandharipande95, AbramovichVistoli02} for details and constructions. 

The space $\overline{{\mathcal M_0}}(\mathcal X/B, \alpha)$ parametrizes maps $f : C\ra X$ from genus zero stable curves to fibers $X=\mathcal X_b$ of $\pi$ such that $f_*[C]=\alpha_b$. The map $\overline{{\mathcal M_0}}(\mathcal X/B, \alpha)\ra B$ is proper. If $f$ is a stable map, we denote by $[f]$ the corresponding point of the Kontsevich moduli stack.

For the remainder of this section, let $X$ be a smooth projective irreducible holomorphic symplectic variety of dimension $2n$ and let $f : C\ra X$ be a map from a stable curve $C$ of genus zero to $X$. We assume furthermore $f$ is unramified at the generic point of each irreducible component of $C$. Let $\mathcal X\ra B$ be a smooth projective morphism of smooth connected quasi-projective varieties and let $0$ be a point of $B$ such that $\mathcal X_0=X$. Let $\mathcal \alpha$ be a global section of Hodge type $(2n-1,2n-1)$ of $R^{4n-2}\pi_*\Z$ such that $\alpha_0=f_*[C]$ in $H^{4n-2}(X, \Z)$.

\begin{prop}\label{prop:defcurves}
Let $ M$ be an irreducible component of $\overline{ M_0}(X, f_*[C])$ containing $[f]$. Then the following holds:
\begin{enumerate}
\item the stack $ M$ has dimension at least $2n-2$;
\item if $ M$ has dimension $2n-2$, then any irreducible component of the Kontsevich moduli stack $\overline{{\mathcal M_0}}(\mathcal X/B, \alpha)$ that contains $M$ dominates $B$;
\end{enumerate}
\end{prop}

In other words, when the assumption of (2) above holds, the stable map $f : C\ra X$ deforms over a finite cover of $B$. Related results have been obtained by Ran, see for instance Example 5.2 of \cite{Ran95}. See also \cite[Theorem 4.1]{AmerikVerbitsky14} for an alternative proof.

\begin{proof}
Let $\widetilde{\mathcal X}\ra S$ be a local universal family of deformations of $X$ such that $B$ is the Noether-Lefschetz locus associated to $f_*[C]$ in $S$. In particular, $B$ is a smooth divisor in $S$. 

Lemma 11 of \cite{BHT} and our hypothesis on $f$ show that we have an isomorphism
$$\mathbb R\mathcal{H}om(\Omega_f^\bullet, \mathcal  O_C)\simeq  N_f[-1]$$
in the derived category of coherent sheaves on $C$, for some coherent sheaf $N_f$. As a consequence, standard deformation theory shows that any component of the deformation space of the stable map $f$ over $S$ has dimension at least
$$\dim{S}+H^1(\mathbb R\mathcal{H}om(\Omega_f^\bullet, \mathcal  O_C))-H^2(\mathbb R\mathcal{H}om(\Omega_f^\bullet, \mathcal O_C))=\dim{S}+\chi(N_f)=\dim{B}+2n-2,$$
the latter equality following from the Riemann-Roch theorem on $C$ and the triviality of the canonical bundle of $X$.

Since the image in $S$ of any component of the deformation space of the stable map $f$ is contained in $B$, the fibers of such a component all have dimension at least $2n-2$. If any fiber has dimension $2n-2$, it also follows that the corresponding component has to dominate $B$, which shows the result.
\end{proof}

The result above holds for any smooth projective variety $X$ with trivial canonical bundle. In order to use it, we need to study the locus spanned by a family of rational curves. The following gives a strong restriction on this locus, and makes crucial use of the symplectic form on $X$. 

\begin{prop}\label{prop:mumford}
Let $X$ be a projective manifold of dimension $2n$ endowed with a symplectic form, and let $Y$ be a closed subvariety of codimension $k$ of $X$. Let $W\subset X$ be a subvariety such that any point of $Y$ is rationally equivalent to a point of $W$. Then the codimension of $W$ is at most $2k$.
\end{prop}

For a similar result see \cite[Theorem 4.4]{AmerikVerbitsky14}.

Before proving the proposition, we first record the following fact from linear algebra.

\begin{lem}\label{fact:sympl}
Let $(F,\omega)$ be a symplectic vector space of dimension $2n$ and 
$V$ a subspace of codimension $k$ of $F$. Then $V$ contains a subspace $V'$ of codimension at most $2k$ in $F$ such that the restriction $\omega_{|V'}$ of the $2$-form  is symplectic on $V'$. In particular, $\omega^{2n-2k}_{|V}\neq 0$.
\end{lem}

 \begin{proof}[Proof of Lemma \ref{fact:sympl}] Let $V^{\perp}$ be the orthogonal to $V$ with respect to the symplectic form $\omega$. Since $\omega$ is non-degenerate, we have
 $$\dim(V^{\perp})=k, $$
which implies that $\dim(V\cap V^{\perp})\leq k$. 

Let $V'$ be a subspace of $V$ such that $V'$ and $V\cap V^{\perp}$ are in direct sum in $V$. Then $\dim(V')\geq 2n-2k$. Furthermore, any $v\in V$ is orthogonal to $V\cap V^\perp$, so that $(V')^\perp=V^{\perp}$. As a consequence, $V'\cap (V')^{\perp}=V'\cap V^\perp=0$, and the restriction of $\omega$ to $V'$ is non degenerate.
 \end{proof}

Let us recall the following result, proved by Voisin \cite[Lemma 1.1]{Voi15} as an application of Mumford's theorem on $0$-cycles.

\begin{lem}\label{lem:mumford-voisin} Let $f : Z\ra X$ be a morphism between smooth projective varieties. Assume that there exists a surjective morphism $p : Z\ra B$ to a smooth projective variety with the property that any two points of $Z$ with the same image under $p$ are mapped by $f$ to rationally equivalent points of $X$.  

Then for any holomorphic form $\eta$ on $X$ there exists a holomorphic form $\eta_B$ on $B$ such that 
$$
 f^*\eta = p^* \eta_B.
$$
\end{lem} 

\begin{proof}[Proof of Proposition \ref{prop:mumford}]
Assume by contradiction that $W$ has dimension at most $2n-2k-1$. We argue as in \cite[Proof of Theorem 1.3]{Voi15}. If $w$ is any point of $W$, define 
$$O_w:=\{x\in X : x\equiv_{rat} w\},$$
where $\equiv_{rat}$ denotes rational equivalence. Then $O_w$ is a countable union of subvarieties of $X$, and a dimension count shows that for any $w$, $O_w$ contains a component of dimension at least $k+1$.

By the countability of Hilbert schemes, there exists a generically finite cover $\phi: B\to W$, a family $f:Z\to B$ of varieties of dimension $k+1$ and a morphism $f:Z\to X$ mapping 
every fiber of $p$ generically finitely onto points which are all rationally equivalent in $X$. We can assume that $B$ and $Z$ are smooth and projective. Lemma \ref{lem:mumford-voisin} shows that for any holomorphic form 
$\eta$ on $X$ there exists a holomorphic form $\eta_B$ on $B$ such that 
\begin{equation}\label{eq:contr}
f^*\eta = p^* \eta_B.
\end{equation}  
On the other hand, if $\omega$ is the symplectic form on $X$, then, by Lemma \ref{fact:sympl}, the $(2n-2k)$-holomorphic form $\eta=\omega^{n-k}$  verifies
$$
 f^*(\omega^{n-k})\not=0. 
$$ 
As $\dim(B)=\dim(W)<2n-2k$ the latter contradicts equation (\ref{eq:contr}).
\end{proof}

The results above allow us to give a simple criterion for the existence of uniruled divisors on polarized deformations of a given holomorphic symplectic variety $X$.  
For similar results see \cite[Corollaries 4.5 and 4.8]{AmerikVerbitsky14}.


%

\begin{cor}\label{cor:divisor}
Let $f: C\to X$ be genus zero stable curve in $X$, and let $\alpha=f_*[C]$. Assume that $f$ is ruling. Then the following holds.
\begin{enumerate}
\item There exists an irreducible component of $\overline{\mathcal M_0}(\mathcal X/B, \alpha)$ containing $[f]$ that dominates $B$. In particular, the stable map $[f]$ deforms over a finite cover of $B$;
\item for any point $b$ of $B$, the fiber $\mathcal X_b$ contains a uniruled divisor $D$ whose cohomology class is a positive multiple of the Poincar\'e dual of $\alpha$, and such that all of the irreducible components of $D$ are ruled by a curve of class $\alpha$.
\end{enumerate}
\end{cor}

\begin{proof}
Let $ M$ be an irreducible component of $\overline{ M_0}(X, f_*[\mathbb P^1])$ containing $[f]$ such that, denoting by $Y$ be the subscheme of $X$ covered by the deformations of $f$ parametrized by $ M$, $Y$ is a divisor in $X$.

Let $ C \to   M$ be the universal curve. By Proposition \ref{prop:defcurves}, the dimension of $ M$ is at least $2n-2$. We claim that equality holds, which implies (1), again by Proposition \ref{prop:defcurves}. Assume by contradiction that $\dim( M)>2n-2$. Since $\dim(Y)=2n-1$, this implies that any fiber of the evaluation map $C\to Y\subset X$ is at least $1$-dimensional, which in turn shows that there exists a subvariety $W\subset X$ of dimension at most $\dim{Y}-2=2n-3$ such that any point of $Y$ is rationally equivalent to a point in $W$. Proposition \ref{prop:mumford} provides the contradiction. Moreover each irreducible component of such divisor is ruled by curves of class $\alpha$.

To show statement (2), it suffices to consider the case where $B$ has dimension $1$ and passes through a very general point of the Noether-Lefschetz locus associated to $\alpha$. Let $\mathcal M $ be an irreducible component of $\overline{{\mathcal M_0}}(\mathcal X/B, \alpha)$ containing $ M$. Then $\mathcal M$ dominates $B$. Let $\mathcal Y\subset \mathcal X\ra B$ be the locus in $\mathcal X$ covered by the deformations of $f$ parametrized by $\mathcal M$. Since $\mathcal M$ dominates $B$, any irreducible component of $\mathcal Y$ dominates $B$. Since the fiber of $\mathcal Y\ra B$ over $0$ is a divisor in $\mathcal X_0=X$, the fiber of $\mathcal Y\ra B$ at any point $b$ has codimension $1$. 

By construction, all the irreducible codimension $1$ components of $Y_b$ are ruled by a curve of class $\alpha$. Note that we cannot deduce that the class of the divisor $Y_b$ inside $\mathcal X_b$ equals that of $Y$
because $\mathcal Y_0$ contains $Y$ and the inclusion could be proper.

Let $b$ be very general in the Noether-Lefschetz locus. The N\'eron-Severi group of $\mathcal X_b$ has rank at most $2$, and it is generated over $\Q$ by the Poincar\'e dual $\alpha^\vee$ of $\alpha$ and the class of the polarization. In particular, the class of the divisor $Y_b$ is a linear combination of $\alpha^\vee$ and the class of the polarization. As a consequence, so is the class of $Y_0$. This holds for any choice of a polarization on $X$.

First assume that $b_2(X)\geq 5$, so that $h^{1,1}(X)\geq 3.$ Assume that $\alpha^\vee$ is not proportional to the class of $Y_0$. Then the Picard number of $X_b$ is $2$ for general $b$, and we may choose $B$ in such a way that there exists $b'$ in $B$ such that the N\'eron-Severi group of $NS(X_{b'})$ has rank at least $3$. We may replace $X$ with $X_{b'}$ and assume that $NS(X)$ has rank at least $3$.

Let $h$ and $h'$ be two ample classes on $X$ such that $\alpha^\vee, h$ and $h'$ span a subspace of rank $3$ of $NS(X)$. Then the class of the divisor $Y_0$ is both a linear combination of $\alpha^\vee$ and $h$ and of $\alpha^\vee$ and $h'$. As a consequence, it is proportional to $\alpha^\vee$.

%
%
%

To give a general argument that covers the case $b_2=4$, we may replace the moduli space of stable genus 0 curves -- for which classical references require projectivity assumptions -- with the (relative) Douady space parametrizing rational curves, and work with non projective manifolds as well. In this framework, we may  deform $X$ with the class $\alpha$ and the uniruled divisor under consideration to an irreducible holomorphic symplectic variety with N\'eron-Severi group generated by $\alpha^\vee$, which proves the proportionality result.

Knowing that the cohomology class of $Y_b$ is proportional to the Poincar\'e dual of $\alpha$, let $\lambda$ be the rational number such that
$$[Y_b]=\lambda [\alpha]^\vee.$$
Then $\lambda$ is independent of $b$. To prove that $\lambda$ is positive, we may assume that $X_b$ is projective by choosing a suitable deformation of the pair $(X, \alpha)$. Let $H$ be an ample class on $X_b$. Then $q(H, Y_b)>0$ and $H.\alpha>0$, so that the coefficient of proportionality is positive.

\end{proof}

%

%
\section{Examples, proof of Theorem \ref{thm:main} and further remarks}
%

In this section we will use the notation and the basic facts recalled in Section 2.1. 
We start by constructing examples of uniruled divisors using the classical 
Brill-Nother theory applied to the desingularizations of curves on a $K3$ surface. 
Then we will present the proof of Theorem \ref{thm:main}. We will also compare
conditions (i) and (ii) in Remark \ref{rmk:precisestate} with the condition appearing in \cite[Corollary A.3]{zal}. 
Finally we will show the existence of  codimension 2
coisotropic subvarieties on projective holomorphic symplectic manifold of \kntiposp in some of the cases where Theorem \ref{thm:main} does not provide uniruled divisors.

\subsection{Examples}\label{ss:examples}

Let $n$ be a positive integer, and let $k$ be an integer between $1$ and $n$. Write $g=k-1$.
Let  $(S,H)$ be a general  polarized $K3$ surface, $m>0$ an integer and  $L:=mH$ with $p_a(L):=p\geq g$.  
Consider the Severi variety parametrizing nodal genus $g$ 
curves inside $|L|$. It is a locally closed subvariety of $|L|$ (see e.g. \cite{sernesi2007deformations} for the basic facts on Severi varieties). 

Recall that the Severi variety on regular surfaces has the expected codimension (equal to the number of nodes)
whenever non-empty (see e.g. \cite[Example 1.3]{CS}). In the case of multiples of the hyperplane section on a general $K3$ surface
non-emptiness has been shown by \cite{Chen02} (see also \cite{GK} for another proof and a generalization). The case $m=1$ corresponds to the ``classical''  Bogomolov-Mumford theorem (see e.g. \cite[Ch. VIII, Section 23]{Barthetal}).

Consider a $g$-dimensional family $\mathcal C_T\to T,\ T\subset |L|$, of curves of geometric genus $g$ whose general member is a nodal curve. By the above $T$ is given by an irreducible component
of the Severi variety parametrizing  
curves with $(p-g)$-nodes inside $|L|$.

It is well-known that the relative symmetric product ${\mathcal C}_T^{(g+1)}$ is uniruled. Indeed for each $t\in T$  the symmetric product $ C_t^{(g+1)}$ is uniruled
as the Abel-Jacobi morphism $AJ_{g+1}: \tilde C_t^{(g+1)}\to \Pic^{g+1} (\tilde C_t)$ onto the $g$-dimensional Picard variety of the desingularization $\tilde C_t\to C_t$
has fibers isomorphic to $\mathbb P H(\tilde C_t,L)$, for all $L\in \Pic^{g+1} (\tilde C_t)$.

Its dimension is $2g+1$.  We have the following key result. 

\begin{prop}\label{prop:mainK3}
Let $S$ and $L$ be as above. 
Then $S^{[g+1]}$ contains a uniruled divisor
which is the image of ${\mathcal C}_T^{(g+1)}$ under the natural rational map 
$${\mathcal C}_T^{(g+1)}\dashrightarrow S^{[g+1]}$$
for a certain irreducible component $T$
of the Severi variety parametrizing  
curves with $(p-g)$-nodes inside $|L|$.
\end{prop}
\begin{proof}
We may assume that $\Pic(S)=\mathbb{Z} H$. We will prove this statement by induction on $g$. 
For  $g=0$, the statement is simply the existence of nodal rational curves proved by \cite{Chen02}. 


It is sufficient to show the claim on the symmetric product $S^{(g+1)}$ of $S$. 
More precisely, we will prove the following statement:
there exists an irreducible component $V$ of the Zariski closure of the Severi variety parametrizing nodal genus $g$ 
curves inside $|L|$ such that, if $\mathcal C_V\to V$ denotes the universal curve and $\mathcal C^{(g+1)}_V\to V$ the relative symmetric product, the natural  morphism
$$
\mathcal C^{(g+1)}_V\to S^{(g+1)}
$$
is generically finite onto its image. Note that this is equivalent to saying that $(g+1)$-generic points on a generic curve of the family lie only on a finite number of
curves of the family. 
 
Indeed as 
$$
\dim \mathcal C^{(g+1)}_V= \reldim (\mathcal C^{(g+1)}_V) + \dim V= (g+1)+g=2g+1
$$
 it follows that the image is a divisor inside $S^{(g+1)}$, and such divisor is uniruled as observed above. 

Note also that the positive dimensional fibers of the morphism $\mathcal C^{(g+1)}_V\to S^{(g+1)}$ cannot lie in a fiber of  $\mathcal C_V^{(g+1)}\to V$,
as $C_t^{(g+1)}$ injects into $S^{(g+1)}$ for every $t\in V$.

By inductive hypothesis, there exists an irreducible component $W$ of the (Zariski closure of the) Severi variety parametrizing nodal genus $g-1$ 
curves inside $|L|$ such that, if $\mathcal C_W\to W$ denotes the universal curve and $\mathcal C^{(g)}_W\to W$ the relative symmetric product, the natural  morphism
$$
\mathcal C^{(g)}_W\to S^{(g)}
$$
is generically finite onto its image. 

Now let $V$ be the Zariski closure of an irreducible component of the Severi variety of nodal genus $g$ curves in $|H|$ obtained by smoothing one node of the curves in $W$ (which can be done again by the regularity of the Severi variety, \cite[Example 1.3]{CS}). 
By construction $W\subset V$.
Let $\mathcal C_V\to V$ be the universal curve. Its restriction over $W$ yields a map $\mathcal C_W\to W$. 
Let $D$ be the image of the morphism
$$
\mathcal C_V^{(g+1)}\to S^{(g+1)}.
$$
Observe that $D$ contains the image $D_{W}$ of 
$$
\mathcal C^{(g+1)}_W\to S^{(g+1)}.
$$
We claim that by the inductive hypothesis $D_{W}$ has codimension 2, or, equivalenty, that the morphism $\mathcal C^{(g+1)}_W\to S^{(g+1)}$ is generically finite onto its image. 
Indeed if $\xi=x_1+\ldots+x_{g+1}$ is a generic point of the image, then, say, $x_1+\ldots+x_{g}$ is a generic point of the 
image of the morphism $\mathcal C^{(g)}_W\to S^{(g)}$. By the inductive hypothesis the points $x_1,\ldots, x_{g}$ lie on finitely many curves of the family $W$, {\it a fortiori} that will be true for
$x_1,\ldots, x_{g}, x_{g+1}$ and the claim follows. 

We want to prove that $D$  contains $D_{W}$ strictly. If this were not the case, by irreducibility, we would have $D=D_{W}$. 
Let $U\subset D$ be an open subset over which the morphisms $\mathcal C^{(g+1)}_W\to S^{(g+1)}$ and $\mathcal C_V^{(g+1)}\to S^{(g+1)}$
are smooth and let $p_1+p_2+\dots +p_{g+1}$ be a  point in $U$. Let $C$ be a nodal genus $g$ curve in $V$ containing these points. 
Let us fix the first $g$ points $p_1,\ldots,p_g$. By induction these points are contained inside a finite number of curves of genus $g-1$ belonging to $W$. 
Let $B_1,\dots, B_m$ be all such curves. Let $U_C\subset C$ be an open subset such that for all $q\in U_C$ we have
$p_1+\ldots+p_g+q\in U$. As we have seen above $p_1,\ldots,p_g,q$ lie on finitely many curves of genus $g-1$ belonging to $W$, and these curves must be
$B_1,\dots, B_m$. Therefore, as $q$ varies in $U_C$, we deduce that $U_C$ is  a subset of a finite union of genus $(g-1)$ curves. As $C$ is irreducible, there is an $i$ such that $C= B_i$, which is clearly a contradiction. Therefore $D$ must strictly contain $D_{W}$ and be a divisor, which is necessarily uniruled.

\end{proof}

%

\begin{rmk}\label{rmk:ndiv}
\rm{ Let $n\geq 2$ be an integer. Let $S$ be a general projective $K3$ surface, $H$ an ample divisor on it and 
$L=mH,\ m\geq 1$ a line bundle with $p_a(L)\geq n-1$. 
For all $k=1,\ldots, n$ consider a $(k-1)$-dimensional family $T\subset |L|$ of nodal curves of geometric genus $k-1$. 
Then consider the closure of the image of the rational map
 $$ {\mathcal C}^{(k)}_T + S^{(n-k)}\dashrightarrow S^{[n]}. 
 $$ 
 By Proposition \ref{prop:mainK3} 
 we obtain $n$ distinct uniruled divisors $D_1,\ldots, D_n$ inside $S^{[n]}$.  }
 \end{rmk}
 
Let us now compute the class of the curve in the ruling of the above divisors. Let us
 denote by $\mathbb P^1_{\mathfrak g^1_k}\subset S^{[k]}$ the image of a rational curve associated to one of the $\mathfrak g^1_k$'s given, as in the proof of Proposition \ref{prop:mainK3}, by any $k$-points on the curve. 
We remark that the hypotheses  imply that for a general choice, these linear series are simple and the nodes of the curve are non neutral (a node $p\in C$ is said to be non neutral with respect to a linear series ${\mathfrak g^1_k}$ on the desingularization $\nu:\tilde C\to C$ if ${\mathfrak g^1_k}(-\nu^{-1}(p))=\emptyset$, i.e. if the 2 points above the node do not belong to the same fiber of the morphism associated to the linear series).
Then, the Riemann-Hurwitz formula (see e.g. \cite[Section 2]{CK}) yields 
\begin{equation}\label{eq:S^k}
[\mathbb P^1_{\mathfrak g^1_k}]= mh - 2(k-1)r_k,
\end{equation}
where $mh$ is the class of $L$ in the N\'eron-Severi group of $S$.\\
 If we add $(n-k)$-distinct generic points $\eta=q_1+\ldots+q_{n-k}$ to $\mathbb P^1_{\mathfrak g^1_k}\subset S^{[k]}$ we get a rational curve inside $S^{[n]}$ which we  denote by $R_k$. 


\begin{prop}\label{prop:class1}
Let $k$ be an integer between $1$ and $n$. Then 
\begin{enumerate}
\item The class of $R_k\subset S^{[n]}$ in $H_2(S^{[n]},\Z)$ is $h-(2k-2)r_n$;
\item The class of $D_k\subset S^{[n]}$ in $H^2(S^{[n]},\Z)$ is proportional to $(2n-2)h-(2k-2)\delta_n$.
\end{enumerate}
\end{prop}

\begin{proof}
(1) Write $R_k = ah-b r_n$. As $R_k=\mathbb P^1_{\mathfrak g^1_k}+\eta$, the intersection product
$R_k\cdot h$ equals $2p-2$, from which, using the facts recalled in Section 2.1 we deduce $a=1$. Again for the choice of $\eta$ we have that $R_k\cdot \delta_n=\mathbb P^1_{\mathfrak g^1_k}\cdot \delta_k$. From this, from (\ref{eq:S^k}) and from Section 2.1 we deduce that
$b=2(k-1)$. This proves the first statement. 

(2) For the second statement we argue similarly as follows. Write $D_k = ah-b \delta_n$. Let $x_1, \ldots, x_{n-1}\in S$ be general points and $C\in |L|$ a general curve. Set $\xi :=x_1+ \ldots+ x_{n-1}$ and consider the curve in $S^{[n]}$ given by $\xi+C$. Such a curve has class $h\in H_2(S,\Z)$. We first describe all the points in the intersection  between $D_k$ and $\xi+C$. Let $I$ be a subset of $k-1$ indices among  $\{1,\ldots, n-1 \}$ and let $\xi_I$ be the corresponding $0$-dimensional subscheme of length $k-1$.  Notice that, as shown in Proposition \ref{prop:mainK3} there exists a finite subfamily $T_\xi$ of cardinality $M$ given by curves in $T$ passing through $\xi_I$. For each  curve $C'\in T_\xi$ and each point $q$ of the $2p-2$ intersection points between $C'$ and $C$ we get an intersection between $D_k$ and $\xi+C$ given by $q+\xi$. All in all we obtain that 
$$
D_k \cdot h = M {{n-1}\choose{k-1}} (2p-2)
$$
from which we deduce that $a=M {{n-1}\choose{k-1}}$. 

Consider $p, x_1,\ldots, x_{n-2}\in S$ general points. Set $\xi=x_1+\ldots +x_{n-2}$ and consider the curve $\P T_p(S)+\xi$ which has class $r_n$. We describe now all the points in the intersection  between $D_k$ and $\P T_p(S)+\xi$. Let $I$ be a subset of $k-2$ indices among  $\{1,\ldots, n-2 \}$ and let $\xi_I$ be the corresponding $0$-dimensional subscheme of length $k-2$. As above there exists a finite subfamily $T_{p+\xi}$ of cardinality $M$ given by curves in $T$ passing through $p+\xi_I$. For each  curve $C'\in T_{p+\xi}$  we get an intersection between $D_k$ and $r_n$ given by $p_{C'}+\xi$, where $p_{C'}$ is the length two $0$-dimensional subscheme supported on $p$ and determined by the tangent direction of $C'$ at $p$. We therefore have 
$$
D_k \cdot r_n = M {{n-2}\choose{k-2}}
$$
from which we deduce that $b=M {{n-2}\choose{k-2}}$. 
Therefore
$$
D_k  = M {{n-2}\choose{k-2}} \frac{1}{(k-1)} \Big{(} (n-1)h - (k-1)\delta_n\Big)
$$
and we are done.
\end{proof}

To account for other polarization types, we proceed as follows. For  $k=2,\ldots, n$
let $\xi' \in S^{[k]}$ be a 0-dimensional non-reduced subscheme corresponding to a ramification point
of a $\mathfrak g^1_k$ on a general curve $\tilde C_t$. Let $\xi \in S^{[n]}$ be a 0-dimensional subscheme obtained by adding $(n-k)$ distinct generic points to $\xi'$. 
Let $\mathbb P^1_\xi$ the exceptional rational curve passing through $\xi$. Consider the 
curve 
$$
R'_k:=R_k \cup \mathbb P^1_\xi
$$
obtained by glueing along $\xi$  the curve $R_k$ corresponding to the $\mathfrak g^1_k$ and the  exceptional rational curve $\mathbb P^1_\xi$. Working with families of rational curves, we obtain as before uniruled divisors $D'_2, \ldots, D'_n$. The divisor $D'_k$ is the union of $D_k$ with the exceptional divisor of $S^{[n]}$. As a direct consequence of Proposition \ref{prop:class1}, we can compute the relevant cohomology classes.

\begin{prop}\label{prop:class2}
Let $k$ be an integer between $2$ and $n$. Then 
\begin{enumerate}
\item The class of $R'_k\subset S^{[n]}$ in $H_2(S^{[n]},\Z)$ is $h-(2k-1)r_n$;
\item The class of $D'_k\subset S^{[n]}$ in $H^2(S^{[n]},\Z)$ is proportional to $(2n-2)h-(2k-1)\delta_n$.
\end{enumerate}
\end{prop}
 
 We now move on to the result we alluded to in  Remark \ref{rmk:precisestate} the introduction.
 
 \begin{thm}\label{thm:precise}
 Let $(X,H)$ be a polarized holomorphic symplectic variety of $K3^{[n]}$-type.
Suppose there exist 
integers $p,g$ and $\epsilon$ such that $p\geq g$ and $\epsilon =0$ or $1$ such that the following two conditions hold:
\begin{enumerate}
\item[(i)] the class $\alpha:=\frac{H^\vee}{\div(H)}\in H_2(X,\mathbb Z)$ can be written as $\gamma +(2g-\epsilon)\eta$, with 
$\eta$ in the monodromy orbit of the class of the exceptional curve on a $K3^{[n]}$ and $\gamma\in \eta^\perp$;
\item[(ii)] $q(\gamma)=2p-2$ (hence $q(\alpha)=2p-2-\frac{(2g-\epsilon)^2}{2n-2}$).
\end{enumerate}
Then there exists an integer $m>0$ such that the linear system $|mH|$ contains a uniruled divisor covered by rational curves of primitive class
equal to $\alpha$.

 \end{thm}
 \begin{proof}
Let $\alpha\in H_2(X, \Z)$ as in the statement of the theorem. 
By Corollary \ref{cor:control-of-polarization-dual}, there exists a $K3$ surface $S$ as above such that the pair $(X, \alpha)$ is deformation equivalent to $(S^{[n]}, [R_k])$ for some $k$ between $1$ and $n$, or to $(S^{[n]}, [R'_k])$ for some $k$ between $2$ and $n$ (with $R_k$ and $R'_k$ as in Proposition \ref{prop:class1} and \ref{prop:class2}). Notice that being deformation equivalent means that $(X,\alpha)$ and $(S^{[n]}, R_k)$ (or $(S^{[n]}, [R'_k])$) are connected by a family $f: \mathcal X\to B$, where B is an irreducible curve, so that the parallel transport brings $\alpha$ to $[R_k]$ (resp. $[R'_k]$). 

Corollary \ref{cor:divisor} shows indeed that $X$ contains a uniruled divisor with class a multiple of $h$ and the theorem now follows.
\end{proof}
 
 \subsection{Finiteness of the exceptions and proof of the main theorem}\label{ss:exceptions}
In this paragraph we prove that, for every dimension, there is at most a finite number of components of the moduli space of polarized irreducible holomorphic symplectic manifolds $(X,H)$ of \kntiposp 
where the strategy of the previous sections does not work.  Together with Theorem \ref{thm:precise} this will conclude the proof of Theorem \ref{thm:main}.

The uniruled divisors we constructed in the previous paragraph  have class $h_S-(2g)r_n$ (or $h_S-(2g-1)r_n$), 
with $2p-2=h_S^2$ and $h_S$ the primitive polarization on the $K3$ surface. We have the following:
\begin{prop}\label{prop:curvegrandi}
Let $C$ be a primitive class of a curve on a manifold of \kntiposp such that its square $C^2$ with respect to the Beauville-Bogomolov form is $>0$. 
If $q(C)\geq n-1$, then $C$ is deformation equivalent to the class of one of the curves constructed in the previous section. 
\end{prop} 
\begin{proof}
We know by Corollary \ref{cor:control-of-polarization-dual} that $C$ is deformation equivalent to either $h_S-2gr_n$ or $h_S-(2g-1)r_n$, with $2g\leq n-1$ (resp. $2g\leq n$). If $p$ were $<g$, 
then, in both cases, we would have
$$
n-1 \leq q(C)=
q(h_S) -4g^2 \frac{1}{2(n-1)} =  2(p-1) -4g^2 \frac{1}{2(n-1)}< 2(g-1)-4g^2 \frac{1}{2(n-1)}\leq n -2
$$ 
which is a contradiction.
\end{proof}
\begin{proof}[Proof of Theorem \ref{thm:main}]
The components of $\mathfrak{M}$ are in bijective correspondence with the monodromy orbits of a given class of positive square in $H^2(X,\mathbb{Z})$, see \cite[Cor. 2.4]{Apo}. 
 For a fixed square of $h$, there is a finite number of orbits (cf. \cite[Proposition 1.2]{GritskenkoHulekSankaran10}), so it follows that if $X$ has a uniruled divisor when $q(h)$ is big enough, our claim will hold.  Let $C$ be a curve class in $H_2(X,\mathbb{Z})$ such that $C=h/div(h)$ under the map
(\ref{eq:dualset}).
 The divisibility of $h$ is at most $2n-2$, therefore if $q(h)\geq (2n-2)^2(n-1)$ the curve class $C$ has square at least $n-1$, so that Proposition \ref{prop:curvegrandi} applies and both items of the theorem follow from Theorem \ref{thm:precise}.
\end{proof}
\begin{cor}\label{cor:rho2}
Let $X$ be a projective irreducible holomorphic symplectic variety of \kntiposp with Picard rank at least two. Then $X$ has an ample  divisor ruled by primitive rational curves.
\end{cor}
\begin{proof}
Since $X$ is projective and has Picard rank at least two, its Picard lattice is indefinite and contains primitive elements of positive arbitrary Beauville-Bogomolov square. The same holds for ample classes. Let $h$ be an ample divisor such that $q(h)\geq (2n-2)^2(n-1)$. Let $C$ be a curve class in $H_2(X,\mathbb{Z})$ such that $C=h/div(h)$ under the map
(\ref{eq:dualset}). As in the proof of Theorem \ref{thm:main} it follows that $q(C)\geq n-1$ and Proposition \ref{prop:curvegrandi} yields our claim.
\end{proof}

\begin{rmk}\label{rmk:npiccolo}
{\em The estimate of Proposition \ref{prop:curvegrandi} is definitely not sharp, indeed  all primitive curves of positive square on irreducible holomorphic symplectic manifolds of \kntiposp with $n\leq 7$ are deformation equivalent to the curves we construct. Indeed, by Corollary \ref{cor:control-of-polarization-dual} we can suppose that our pair is $(S^{[n]},h_S-\mu r_n)$ with $0\leq \mu\leq n-1$ and $S$ is a K3 of genus $p$. This curve is constructed from a curve of class $h_S-2g r_n$ with the eventual addition of a tail of class $r_n$, so that $2g\leq n$. Let us suppose that $g>p$ and $n\leq 7$. We have 
$$
 q(h_S-2g r_n)=2p-2-2\frac{g^2}{n-1}\leq 2p-2-2\frac{(p+1)^2}{n-1}\leq 2p-2-2\frac{(p+1)^2}{6}.
$$
However, the last value is never positive, hence $q(h_S-2g r_n)$ can only be negative. Analogously, for $C=h_S-(2g-1)r_n$, we have $q(C)\leq \frac{20p-25-4p^2}{12}$ with $g\geq p+1$ and $2g\leq n$, which is again not positive.  }
\end{rmk}

 %
 \section{Comparison with the work of Oberdieck, Shen  and  Yin and further results}\label{s:k3no}
%
In this section we compare our results with those of Oberdieck, Shen  and  Yin \cite[Corollary A.3]{zal}. 
We show that conditions (i) and (ii) in Theorem \ref{thm:precise} are precisely satisfied in {\it all the cases} where the obstruction discovered by \cite{zal} does not prevent
uniruled divisors covered by primitive rational curves to exist. In other words
we show that our result is sharp.
We discuss how the counterexamples to the existence of uniruled divisors covered by rational curves of primitive class propagate in higher dimensions. 
Moreover, up to a relatively high value of the dimension ($2n=26$), we show that the existence of uniruled divisors can nevertheless be obtained via 
 non-primitive rational curves. Finally in some of the ``exceptional cases'' (those where existence of uniruled divisors covered by rational curves of primitive class is excluded)
we show that the codimension of the locus covered by the primitive rational curves is two.

 \subsection{Comparison with the work of Oberdieck, Shen  and  Yin}\label{ss:k3no}

To conclude this subsection let us recall the condition in \cite[Corollary A.3]{zal}:
\begin{prop}[\cite{zal}]\label{prop:zalcondition}
Let $\beta$ be a curve class on a manifold $X$ of \kntipo. Then there is a uniruled divisor swept out by $\beta$ if 
\begin{align}
 \beta^2 &= -2+\sum^{n-1}_i 2d_i-\frac{1}{2n-2}(\sum^{n-1}_i r_i)^2,\\
[\beta] &=\pm [\sum{r_i}],\\
4d_i-r_i^2& \geq 0.
\end{align}
Here $[\beta]$ denotes the class of $\beta$ seen as an element of the discriminant group 
$H^2(X,\mathbb{Z})^\vee/H^2(X,\mathbb{Z})$ (with a generator of square $-1/(2n-2)$ which is in the same monodromy orbit of the class $r_n$ of exceptional lines on $S^{[n]}$ for any K3 surface $S$). 
The converse holds if $\beta$ is irreducible. 
\end{prop}

\begin{prop}\label{prop:equivalent}
Let $n,g>0$ be integers such that $2g\leq n$. Let $p$ be an integer number. 
The condition $p\geq g$ for a curve of class $h_S-2g r_n$ with $h_S^2=2p-2$,  on an Hilbert scheme $S^{[n]}$ is equivalent to the conditions in Proposition \ref{prop:zalcondition}.
\end{prop}
\begin{proof}
Let us call $\beta=h_S-2g r_n$. We have $\beta^2=-2+2p-\frac{1}{2n-2}4g^2$ and $[\beta]=[2g r_n]$ in the discriminant group. Therefore, we must have
\begin{align}
\sum d_i&=p,\\
\sum r_i &=2g.
\end{align}
If $p\geq g$, we can set $r_i= 2$ for $g$ indices $i$ such that $d_i\neq 0$ and set $r_i=0$ for all the others, so that the conditions in Proposition \ref{prop:zalcondition} are satisfied. On the other hand, if $g>p$, there is at least one $r_i>2d_i$, so that $4d_i-r_i^2<0$, contradicting the third 
item in Proposition \ref{prop:zalcondition}.
\end{proof}

\begin{rmk}\label{rmk:zal4ever}
{\rm{Observe that if $p,g$ and $n$ are as in Proposition \ref{prop:equivalent}, then for all $n'\geq n$ the integers $p,g$ and $n'$
provide again examples of primitive classes $h_S-2g r_{n'}\in H_2(S^{[n']},\Z)$ which cannot rule a divisor. 
}}
\end{rmk}

\begin{rmk}\label{rmk:nozal}
{\rm{Condition (ii) in Remark \ref{rmk:precisestate} is not sufficient to ensure the existence of a uniruled divisor
covered by primitive rational curves. Indeed let $n=11$ and consider two general polarized $K3$ surfaces $(S_1,h_1)$ and $(S_2,h_2)$ 
respectively of genus $2$ and $4$.  One checks that the classes $C_1:=h_{S_1}-r_{11}$ on $S_1^{[11]}$ and $C_2:=h_{S_2}-9r_{11}$ on $S_2^{[11]}$ have the same square 
$=2- 1/{20}$. The divisors $h_i,\ i=1,2$ such that $C_i=h_i/\div(h_i)$ are $h_1= 20 h_{S_1}-\delta_{11}$ and  $h_2= 20 h_{S_2}-9 \delta_{11}$.
Nevertheless they are not in the same orbit under the monodromy action. This can be seen as follows: by Markman two classes are monodromy equivalent if and only if they have the same square and their images in the discriminant group are equal up to sign. Now the image of the first class is $[r_{11}]$ while that of the second is $[9r_{11}]$. 
Moreover, by Proposition \ref{prop:zalcondition}, only the first one can be the class of rational curves covering a divisor.
}}
\end{rmk}

\subsection{Uniruledness via non-primitive rational curves}\label{ss:nonprim}
Let $\epsilon$ be equal to $0$ or $1$.
Suppose $n,p$ and $g$ are positive integers such that
\begin{eqnarray}
\label{eq:cond1} \frac{(n-1+\epsilon)}{2}\geq g\geq p+1;\\
\label{eq:cond2} (2p-2)(2n-2)> (2g-\epsilon)^2.  
\end{eqnarray}
Let $(S,h_S)$ be a general polarized $K3$ surface of genus $p=p_a(h_S)$. 
Let $C\in H_2(S^{[n]},\Z)$ be a primitive curve class of the form
$$
 h_S - (2g-\epsilon)r_n. 
$$
Notice that condition (\ref{eq:cond2}) is equivalent to $q(C)>0$. By condition (\ref{eq:cond1}) we cannot
apply our main result to $C$ (and by \cite[Corollary A.3]{zal} there is no way to obtain a uniruled divisor covered 
by rational curves of class $C$). Nevertheless it makes sense to ask the following
\begin{ques}\label{ques:multiple}
Does there exist an integer $m>0$ such that $mC$ is represented by rational curves covering a divisor?
\end{ques} 
In particular, as the Severi varieties of $|mh_S|$ are known to be non-empty by \cite{Chen02}, it is natural to try and extend our approach to the multiple hyperplane linear system.
Precisely we can look for an integral nodal curve $C'\in |mh_S|$ of genus $g'=\lceil \frac{2mg-m\epsilon}{2}\rceil$ and take $g'+1$ points
on $C'$ to get a rational curve in $S^{[n]}$ of class $mC$ (possibly after the union of an exceptional tail, depending on the parity of $m$).

The obvious necessary numerical conditions to be satisfied are:
\begin{eqnarray}
\label{eq:cond1'} g'+1\leq n;\\
\label{eq:cond2'} g'\leq p_a(mh_S)= m^2(p-1)+1.  
\end{eqnarray}
If such an integer $m$ exists, by applying Proposition \ref{prop:mainK3} and the same strategy of Theorem \ref{thm:main} we would get existence of uniruled divisors in the components of $\mathfrak M$ (which is  the union  $\cup_{d>0}\mathfrak M_{2d}$ 
of the moduli spaces $\mathfrak M_{2d}$ of projective irreducible holomorphic symplectic varieties of $K3^{[n]}$-type polarized by a 
line bundle of degree $2d$) left out from Theorem \ref{thm:main}. 

Let us define the following quantities (coming from conditions (\ref{eq:cond1'}) and (\ref{eq:cond2'}) above by distinguishing according to the parity of $m$):
\begin{eqnarray}
\label{eq:m_maxeven} m^{even}_{max}:= \frac{2(n-1)}{2g-\epsilon};\\ 
\label{eq:m_maxodd} m^{odd}_{max}:= \frac{2n-3}{2g-\epsilon};\\
\label{eq:m_mineven} m^{even}_{min}:= \frac{g-\epsilon/2+\sqrt{(g-\epsilon/2)^2-4(p-1)}}{2(p-1)};\\
\label{eq:m_minodd} m^{odd}_{min}:= \frac{g-\epsilon/2+\sqrt{(g-\epsilon/2)^2-2(p-1)}}{2(p-1)}.  
\end{eqnarray}

From the above discussion, we deduce the following

\begin{prop}\label{prop:mbound}
If there exists an integer $m>0$ such that 
\begin{equation}\label{eq:mbound}
m_{min}^\bullet \leq m \leq m_{max}^\bullet,
\end{equation}
then Question \ref{ques:multiple} has a positive answer. 
\end{prop}

To apply Proposition \ref{prop:mbound} first of all one must check that  $m_{min}^\bullet \leq m_{max}^\bullet$. 
This easily follows from (\ref{eq:cond2}) when $m$ is even or $\epsilon=0$. 

We can now show that the apparent persistence of the pathologies, observed in Remark \ref{rmk:zal4ever}, can be avoided by taking non-primitive curves:
\begin{prop}
Let $n,p$ and $g$ as above. Then for all $n'\geq g+1+n$ there exists an integer $m>0$  satisfying (\ref{eq:mbound}).
\end{prop}
\begin{proof}
If $n'\geq n+1$ we always have $m_{min}^\bullet \leq m_{max}^\bullet$. Moreover for $n'\geq n+1+g$ the value of $m_{max}^\bullet$
increases by at least one.
\end{proof}
We have seen in Remark \ref{rmk:npiccolo} that we do have existence of uniruled divisors covered by primitive rational curves in sufficiently ample linear systems 
on  irreducible holomorphic symplectic manifolds of \kntiposp with $n\leq 7$. 
 
There are  cases in which condition (\ref{eq:mbound}) is easily checked to hold: for instance for all  $p, g$ and $n$ verifying 
(\ref{eq:cond1}) and (\ref{eq:cond2}) and 
 such that $(2p-2)| (2g-\epsilon)$. In particular this yields the following.
 
\begin{thm}\label{thm:nleq13}
Let $8\leq n\leq 13$ be an integer. Let $\mathfrak M=\cup_{d>0}\mathfrak M_{2d}$  be the union
of the moduli spaces $\mathfrak M_{2d}$ of projective irreducible holomorphic symplectic varieties of $K3^{[n]}$-type polarized by a 
line bundle of degree $2d$.  For all $(X,H)\in \mathfrak M$
there exists a positive integer $a$ such that the linear system $|aH|$ contains a uniruled divisor.
\end{thm}
\begin{proof}
We follow the strategy of taking a curve in the multiple hyperplane system $|mH_S|$ outlined above. 
In the following table, we list the curve classes arising as exceptions to Theorem \ref{thm:main}, together with the genus $p$ of the $K3$ surface $S$, the values of $g$ and $\epsilon$, the minimal $m$ which satisfies \eqref{eq:mbound} and the smallest $n$ such that $h_S-(2g-\epsilon)r_n$ does not satisfy Theorem \ref{thm:main}.  

\begin{table}[hbt!]
\begin{tabular}{|c|c|c|c|c|c|}
\hline
Class & $p$ & $g$ & $\epsilon$ & $m$ & $n$\\
\hline
$h_S-5r_n$ & 2 & 3 & 1 & 2 & 8\\
\hline
$h_S-7r_n$ & 3 & 4 & 1 & 2 & 8\\
\hline
$h_S-8r_n$ & 3 & 4 & 0 & 2 & 10\\
\hline
$h_S-9r_n$ & 4 & 5 & 1 & 2 & 10\\
\hline
$h_S-10r_n$ & 4 & 5 & 0 & 2 & 10\\
\hline
$h_S-6r_n$ & 2 & 3 & 0 & 3 & 11\\
\hline
$h_S-9r_n$ & 3 & 5 & 1 & 2 & 12\\
\hline
$h_S-11r_n$ & 4 & 6 & 1 & 2 & 12\\
\hline
$h_S-11r_n$ & 5 & 6 & 1 & 2 & 12\\
\hline
$h_S-12r_n$ & 5 & 6 & 0 & 2 & 13\\
\hline

\end{tabular}
\end{table}
The conclusion immediately follows from the table.
\end{proof}
%
%

The first case when this strategy does not work appears for $n=14$, by taking 
$$
 C:= h_S-10r_{14},
$$
where $h_S$ is a polarization of genus $p_a(h_S)=3$ and one checks that $2<m_{min}<m_{max}<3$.

The bad news is that, even asymptotically in $n$, there is no hope that condition (\ref{eq:mbound}) can hold as shown by the following. 
\begin{ex}\label{ex:ciaone}
{\rm {Take $g=\lceil \frac{n-1}{3}\rceil$, $p-1=\lceil \frac{n-1}{9}\rceil+1$ and $n$ large enough. Conditions (\ref{eq:cond1}) and (\ref{eq:cond2})
are satisfied. However both $m_{min}^\bullet$ and $m_{max}^\bullet$ are $<3$, but $m_{min}^\bullet \to_{n\to +\infty} 3$.
}}
\end{ex}

\begin{rmk}\label{rmk:BNforse}
{\rm{One may wonder whether a minor modification of this strategy might still lead to the existence of uniruled divisors in all  cases. 
One possibility is to construct {\it different } rational curves coming from the Brill-Noether theory of nodal curves in the multiple
hyperplane linear system of a general $K3$. This approach presents two difficulties. One has first to control the Brill-Noether theory of such nodal curves
(which does not seem to be an easy task, knowing that already smooth curves in multiples of the hyperplane section are not Brill-Noether general). 
Secondly, even if one disposed of a family of $\g^1_n$'s on nodal  curves in $|mh_S|$ of the right dimension, the analogue of Proposition \ref{prop:mainK3}
should still be proved for such family.
}}
\end{rmk}
\subsection{Some codimension 2 coisotropic subvarieties}\label{ss:codim2}

In this paragraph we look at the cases where, by \cite{zal}, there are no uniruled divisors ruled by primitive rational curves, and try to study the codimension of the ruled locus in this case. In particular, we have the following:

\begin{thm}\label{thm:codim2}
Let $X$ be a  polarized irreducible holomorphic symplectic manifold of \kntipo. Let $C$ be a curve class such that the pair $(X,C)$ is deformation equivalent to  $(S^{[n]},h_S-(2g-1)r_n)$, where $q(h_S)=2g-4$, for a certain integer $n/2\geq g> 2$.
 Then $X$ has a codimension two locus covered by rational surfaces ruled by a primitive curve class.
\end{thm}

This will follow from the following result, proven in \cite[Theorem 6.1]{klm2}. We refer the reader to  \cite{klm2} for the notation.
\begin{thm}[Theorem 6.1, \cite{klm2}]\label{thm:KLM2_immagine}
Let $(S,h_S)$ be a very general primitively polarized $K3$ surface  of genus $p:=p_a(h_S) \geq 2$.
Let $0\leq g \leq p$ and $n \geq 2$ be integers satisfying
\begin{equation} \label{eq:nuovobound}
 2(p-g)+2 \leq \chi:=g-n+3 \leq p-g+n+1.
\end{equation} 
Then on $S^{[n]}$ there exists a $(2n-2)$-dimensional family of rational curves of class 
\begin{equation}\label{class}
h_S-(g+n-1)r_n.
\end{equation}
which covers a subvariety birational to a $\mathbb{P}^{\chi-2(p-g)-1}$-bundle on a
 holomorphic symplectic manifold of dimension $2(n+1+2(p-g)-\chi)$. 
\end{thm}
The above theorem applies only to finitely many deformation types of $(X,h)$ for any dimension, but it turns out 
that it can be used to produce coisotropic subvarieties of codimension $\geq 2$ in the exceptions to Theorem \ref{thm:main}. 

\begin{prop}\label{prop:ultima}
Let $(X,h)$ be a polarized irreducible holomorphic symplectic manifold of \kntiposp and let $C$ be the primitive curve class equal to 
$h/div(h)$. Suppose that the conditions in Proposition \ref{prop:zalcondition} are not satisfied. 
Then there exist a polarized $K3$ surface $(S,h_S)$ of genus $p=p_a(h_S)$ and an integer $g\leq  p$
such that:
\begin{enumerate}
\item[(i)] $p$ and $g$ satisfy conditions (\ref{eq:nuovobound});
\item[(ii)] $X$ is deformation equivalent to ${S^{[n]}}$ and the class $C$ is sent to $h_S-(g+n-1)r_n$ by the parallel transport.
\end{enumerate}
\end{prop}
\begin{proof}
By Corollary \ref{cor:control-of-polarization-dual} the pair $(X,C)$ is deformation equivalent to $(\Sigma^{[n]}, h_\Sigma-(2\gamma -\epsilon)r_n )$, where $\Sigma$ is a K3 surface of genus $\pi:=p_a(h_\Sigma)$, $\gamma$ a positive integer and $\epsilon= 0, 1$. As the conditions in Proposition \ref{prop:zalcondition} are not satisfied, by Proposition \ref{prop:equivalent}
we have $p_a(h_\Sigma)< \gamma$.
By Remark \ref{rmk:traslo} we can choose as system of representatives of $\Z/(2n-2)\Z$, up to the action of $-1$, the set $[2n-2, 3n-3]\cap \N$. 
Hence we can deform to a different punctual Hilbert scheme on a polarized K3 surface $(S,h_S)$ such that our curve has class $h_S-(2n-2+2\gamma-\epsilon)r_n$ and the genus $p$ of $S$ is  $p=\pi+n-1+2\gamma-\epsilon$. Set $g:=n-1+2\gamma -\epsilon$.

It follows that $\chi=2\gamma+2-\epsilon$. The conditions we must check are
$$2(p-g)+2\leq \chi\leq (p-g)+n+1.$$
As $p-g=\pi<\gamma$ and $2\gamma-\epsilon \leq n-1$, they are always satisfied and Theorem \ref{thm:KLM2_immagine} applies, giving the desired locus of codimension ${\chi-2(p-g)-1}\geq 2$.
\end{proof}

\begin{proof}[Proof of Theorem \ref{thm:codim2}]
Let $C$ be as in the statement. The conditions of Proposition \ref{prop:zalcondition} are not satisfied, as $p_a(h_S)=g-1$. Hence we can apply 
Proposition \ref{prop:ultima} and get a new pair $(S_1^{[n]}, h_{S_1}-((2g-1+n-1)+n-1)r_n)$ satisfying conditions (\ref{eq:nuovobound}). Therefore we can apply
Theorem \ref{thm:KLM2_immagine} to $(S_1^{[n]}, h_{S_1}-((2g-1+n-1)+n-1)r_n)$ to deduce the existence of a $(2n-2)$-dimensional family of rational curves
covering a coisotropic subvariety. By Proposition \ref{prop:defcurves} these rational curves deform to the initial variety $X$ and by construction they have class equal to $C$. By \cite{zal}, the codimension of the locus covered by their deformations on $X$ cannot be less the $2$. Thus if it is two on 
$S_1^{[n]}$ that must be the case also on a general point $X'$ in the component of the moduli space $\mathfrak M$ containing $(X,C)$.
Observe that in this case we have $\chi=2g+1$ and the locus covered by these rational curves has therefore codimension $2g+1-2g+1=2$.
Hence we have a coisotropic subvariety $Z'\subset X'$ covered by rationally chain connected surfaces $F'$. The flat limit $Z\subset X$ of $Z'$ 
is covered by the flat limits $F$ of the RCC surfaces $F'$ and of course $F$ is RCC. 
The theorem follows.
\end{proof}
\begin{rmk}\label{rmk:ecc}
{\rm{For $n=8,9$, where the first exceptions discovered by \cite{zal} appear, the Theorem above applies. 
This is not the case for the exceptions of classes $h_S-8r_{10}$ and $h_S-10r_{10}$ in dimension $20$ (cf. Theorem \ref{thm:nleq13}).
}}
\end{rmk}

%
\section{Application to $0$-cycles}
%
In what follows we always consider Chow groups with rational coefficients. Throughout this section, let $X$ be an irreducible holomorphic symplectic variety. If $Y$ is a variety, let $CH_0(Y)_{hom}$ be the subgroup of $CH_0(Y)$ consisting of zero-cycles of degree zero.

\begin{definition}
Let $D$ be an irreducible divisor on $X$. We denote by $S_1 CH_0(X)_{D,hom}$ the subgroup  $$S_1 CH_0(X)_{D,hom}:=\Im \big( CH_0(D)_{hom}\to CH_0(X)\big),$$
of $CH_0(X)$. 
We denote by $S_1 CH_0(X)_{D}$ the subgroup 
$$S_1 CH_0(X)_{D}:=\Im \big( CH_0(D)\to CH_0(X)\big)$$
of $CH_0(X)$.
\end{definition}

\begin{lem}\label{lemma:basic-equality-S1}
Let $D$ and $D'$ be two irreducible uniruled divisors on $X$ and $R$ and $R'$ the  general curves in the respective rulings. If $D.R'\neq 0$ and $D'.R\neq 0$,
then $S_1CH_0(X)_D=S_1CH_0(X)_{D'}$ and $S_1CH_0(X)_{D, hom}=S_1CH_0(X)_{D', hom}$
\end{lem}

\begin{proof}
Let $\pi : \widetilde D\ra T$ and $\pi' : \widetilde D'\ra T'$ be rulings of varieties $\widetilde D$ and $\widetilde D'$ mapping finitely to $D$ and $D'$ respectively. The curves $R$ and $R'$ are the images of the general fibers of $\pi$ and $\pi'$ respectively. 
As a consequence of the hypothesis, both projections in the following diagram
\begin{equation}
\xymatrix{
\pi^{-1}(\Sigma)\ar[d]_{\pi_{|\Sigma}}\ar[r]&\Sigma:=D\cap D'&(\pi')^{-1}(\Sigma) \ar[l]\ar[d]^{\pi_{|\Sigma} '}\\
T& & T'}
\end{equation}
are surjective, which implies that  
$$
S_1 CH_0(X)_D= \Im\big ( CH_0(\Sigma)\to CH_0(X)\big)= S_1 CH_0(X)_{D'}.
$$
and
$$
S_1 CH_0(X)_{D, hom}= \Im\big ( CH_0(\Sigma)_{hom}\to CH_0(X)\big)= S_1 CH_0(X)_{D', hom}.
$$
\end{proof}

\begin{proof}[Proof of Theorem \ref{thm:appliA}] We give the proof for $S_1 CH_0(X)_{D}$. The proof for $S_1 CH_0(X)_{D,hom}$ is exactly the same. 

Applying Corollary \ref{cor:divisor}, we may find an ample divisor $H$, all of whose irreducible components $H_i$ are ruled by a rational curve $R_i$ of class $\alpha$, Poincar\'e dual to that of $H$. First note that, since $H$ is ample, for any pair of indices $i,j$ we have
$$
 H_i \cdot R_j =q(H_i,H) >0.
$$ 
Then, by Lemma \ref{lemma:basic-equality-S1} we conclude that the groups $S_1CH_0(X)_{H_i}$ are independent of $i$. 

Let $D$ be an irreducible uniruled divisor on $X$ and $R$ a general curve of its ruling. Since $H$ is ample, we may find an integer $i$ such that $H_i.R\neq 0$. Furthermore, as above, we have $R_i.D\neq 0.$ By Lemma \ref{lemma:basic-equality-S1}, we obtain the equality
$$S_1CH_0(X)_{D}=S_1CH_0(X)_{H_i},$$
which concludes the proof.
\end{proof}
%

Thanks to Theorem \ref{thm:appliA} we can drop the dependence on $D$ from the notation and in what follows, under the same hypotheses, we will simply write $S_1 CH_0(X)$ and $S_1 CH_0(X)_{hom}$ for the groups $S_1CH_0(X)_D$ and $S_1CH_0(X)_{D, hom}$. 


\begin{prop}\label{prop:appli}
Let $X$ be a projective holomorphic symplectic variety, and let $D$ be an irreducible uniruled divisor on $X$. Suppose that $X$ possesses an ample ruling curve $R$. Then
$$
 S_1 CH_0(X)_{hom}= D\cdot CH_1(X)_{hom} \ \ {\textrm{and }}\ 
 S_1 CH_0(X)= D\cdot CH_1(X).
$$
\end{prop}
\begin{proof}
We give the proof for $S_1 CH_0(X)_{hom}$. Basic intersection theory \cite[chapter 6]{FultonIT} guarantees the inclusion 
$$D\cdot CH_1(X)_{hom}\subset \Im(CH_0(D)_{hom}\ra CH_0(X)_{hom}) = S_1 CH_0(X)_{hom}.$$
%

To prove the other inclusion
consider any irreducible uniruled component $H_{i}$ of the ample divisor $H$ ruled by $R$. 
By Theorem \ref{thm:appliA} we have $S_1 CH_0(X)_{D, hom}=S_1 CH_0(X)_{H_{i_0}, hom}$.
Notice moreover that by the hypothesis,
\begin{equation}\label{eq:not0}
D\cdot R = D \cdot \lambda H^\vee = \lambda q(D,H) \not = 0 \textrm{ for some } \lambda\not = 0.
\end{equation}
Consider 
$$Z:=\sum n_k x_k\in S_1 CH_0(X)_{H_{i_0}, hom},
$$
where the $x_k$ lie in $H_{i_0}$. For each $x_k$, let $R_{x_k}$ be a curve in the ruling of $H_{i}$ containing it. Then, by (\ref{eq:not0}), there exists a rational number $\mu>0$, independent of $k$, such that
$$ 
 x_k= \mu D\cdot R_{x_k}. 
$$
Hence 
$$
Z=\sum n_i x_i=\mu D\cdot(\sum n_k R_{x_k})
$$
holds in $CH_0(D)$, from which  we see that 
$$
S_1 CH_0(X)_{D, hom}=S_1 CH_0(X)_{H_{i_0}, hom}\subset D\cdot CH_1(X)_{hom}.
$$
\end{proof}

%

\begin{proof}[Proof of Theorem \ref{thm:appliB}] 
Write $L=\sum m_i D_i$, where $D_i$ is irreducible and uniruled for all $i$.  
We have 
$$
L\cdot CH_1(X)_{hom}\subset \sum D_i\cdot CH_1(X)_{hom}=S_1 CH_0(X)_{hom}
$$
where the last equality holds thanks to Proposition \ref{prop:appli} (notice that we do not automatically have the equality, since some of the $m_i$'s may be negative).

To prove the other inclusion we argue as in Proposition \ref{prop:appli}. 
We take an irreducible uniruled divisor $D$ such that $L\cdot R_D\not =0$, where $R_D$ is a curve in the ruling of $D$. Such a divisor exists by the hypothesis, as we can take any irreducible component of the divisor $H$ ruled by the ample curve $R$. 
Let  $Z:=\sum n_i x_i\in \Im (CH_0(D)_{hom}\to CH_0(X))=S_1CH_0(X)$.
Then the equality
$$
\sum n_i x_i=\lambda L\cdot(\sum n_i D_{x_i})
$$
holds in $CH_0(X)$ for some rational number $\lambda$, hence 
$$
S_1 CH_0(X)_{D}\subset L\cdot CH_1(X)_{hom}.
$$
\end{proof}


%
\section{Some open questions}
%

We briefly discuss some questions raised by our results above. 
Theorem \ref{thm:main} suggests a natural extension to general projective holomorphic symplectic varieties in the following way.

\begin{ques}\label{ques:subvar}
Let $X$ be a projective holomorphic symplectic variety of dimension $2n$, and let $k$ be an integer between $0$ and $n$. Does there exist a subscheme $Y_k$ of $X$ of pure dimension $2n-k$ such that its 0-cycles are supported in dimension $2n-2k$ ?
\end{ques}

The question above has been put into a larger perspective by Voisin, in \cite{Voi15}, as a 
key step towards the construction of a multiplicative splitting in the Chow group. 
In view of our constructions, it seems natural to hope for a positive answer for Question \ref{ques:subvar}. Note that this is the case if $X$ is of the form $S^{[n]}$ for some $K3$ surface $S$ as follows by taking $Y$ to be the closure in $S^{[n]}$ of the locus of points $s_1+\ldots+s_n$, where the $s_i$ are distinct points of $S$, $k$ of which lie on a given rational curve of $S$.

It would be interesting to refine Question \ref{ques:subvar} to specify the expected cohomology classes of the subschemes $Y_k$. 

\bigskip

The particular case of middle-dimensional subschemes seems of special interest in view of the study of rational equivalence on holomorphic symplectic varieties. 

\begin{ques}\label{ques:Lagrangian}
Let $X$ be a projective holomorphic symplectic variety of dimension $2n$. Does there exist a rationally connected subvariety $Y$ of $X$ such that $Y$ has dimension $n$ and nonzero self-intersection ?
\end{ques}

A positive answer to question \ref{ques:Lagrangian} would lead to the existence of a canonical zero-cycle of degree $1$ on $X$, as in the case of $K3$ surfaces. This raises the following question.

\begin{ques}
Assume that Question \ref{ques:Lagrangian} has a positive answer for $X$ and let $y$ be any point of $Y$. Let $H_1, \ldots, H_r$ be divisors on $X$, and let $k_1, \ldots, k_{n}$ be nonnegative integers such that $r+\sum_i 2ik_i =2n$. 

Do we have 
\begin{equation}\label{eq:beauville}
H_1\cdot \ldots \cdot H_r \cdot c_{2}(X)^{k_1}\cdot \ldots \cdot c_{2n}(X)^{k_n}= \deg (H_1\cdot \ldots \cdot H_r \cdot c_{2}(X)^{k_1}\cdot \ldots \cdot c_{2n}(X)^{k_n})\cdot y
\end{equation}
 inside $CH_0(X)$ ?
\end{ques}
The existence of a degree 1 $0$-cycle $c_X$ verifying the equality (\ref{eq:beauville}) above
is a consequence of the Beauville conjecture. 
We wonder whether such a $0$-cycle can be realized  in a geometrically meaningful way as  a point on a rationally connected half-dimensional subvariety. 

Even in the case of a general polarized fourfold of $K3^{[2]}$-type, we do not know the answer to the preceding questions.

\bigskip

Finally, Question \ref{ques:subvar} raises a counting problem as in the case of the Yau-Zaslow conjecture for rational curves on $K3$ surfaces \cite{YauZaslow95}, which was solved in \cite{KlemmMaulikPandharipandeScheidegger10}. We do not know of a precise formulation for this question.

\bibliographystyle{alpha}
\bibliography{CMP}

\begin{thebibliography}{BHPVdV04}

\bibitem[Apo14]{Apo}
Apostol Apostolov.
\newblock Moduli spaces of polarized irreducible symplectic manifolds are not
  necessarily connected.
\newblock In {\em Annales de l'Institut Fourier}, volume~64, pages 189--202,
  2014.

\bibitem[AV02]{AbramovichVistoli02}
Dan Abramovich and Angelo Vistoli.
\newblock Compactifying the space of stable maps.
\newblock {\em J. Amer. Math. Soc.}, 15(1):27--75, 2002.

\bibitem[AV15]{AmerikVerbitsky14}
E.~Amerik and M.~Verbitsky.
\newblock {Rational Curves on Hyperk\"ahler Manifolds}.
\newblock {\em International Mathematics Research Notices},
  2015(23):13009--13045, 05 2015.

\bibitem[Bea83]{Beauville83}
A.~Beauville.
\newblock Vari\'et\'es {K}\"ahleriennes dont la premi\`ere classe de {C}hern
  est nulle.
\newblock {\em J. Differential Geom.}, 18(4):755--782 (1984), 1983.

\bibitem[Bea07]{Beauville07}
Arnaud Beauville.
\newblock On the splitting of the {B}loch-{B}eilinson filtration.
\newblock In {\em Algebraic cycles and motives. {V}ol. 2}, volume 344 of {\em
  London Math. Soc. Lecture Note Ser.}, pages 38--53. Cambridge Univ. Press,
  Cambridge, 2007.

\bibitem[BHPVdV04]{Barthetal}
Wolf~P. Barth, Klaus Hulek, Chris A.~M. Peters, and Antonius Van~de Ven.
\newblock {\em Compact complex surfaces}, volume~4 of {\em Ergebnisse der
  Mathematik und ihrer Grenzgebiete. 3. Folge. A Series of Modern Surveys in
  Mathematics [Results in Mathematics and Related Areas. 3rd Series. A Series
  of Modern Surveys in Mathematics]}.
\newblock Springer-Verlag, Berlin, second edition, 2004.

\bibitem[BHT11]{BHT}
Fedor Bogomolov, Brendan Hassett, and Yuri Tschinkel.
\newblock Constructing rational curves on {K}3 surfaces.
\newblock {\em Duke Math. J.}, 157(3):535--550, 2011.

\bibitem[BM96]{BehrendManin96}
K.~Behrend and Yu. Manin.
\newblock Stacks of stable maps and {G}romov-{W}itten invariants.
\newblock {\em Duke Math. J.}, 85(1):1--60, 1996.

\bibitem[BV04]{BeauvilleVoisin04}
Arnaud Beauville and Claire Voisin.
\newblock On the {C}how ring of a {$K3$} surface.
\newblock {\em J. Algebraic Geom.}, 13(3):417--426, 2004.

\bibitem[Che02]{Chen02}
Xi~Chen.
\newblock A simple proof that rational curves on {$K3$} are nodal.
\newblock {\em Math. Ann.}, 324(1):71--104, 2002.

\bibitem[CK14]{CK}
Ciro Ciliberto and Andreas~Leopold Knutsen.
\newblock On {$k$}-gonal loci in {S}everi varieties on general {$K3$} surfaces
  and rational curves on hyperk\"ahler manifolds.
\newblock {\em J. Math. Pures Appl. (9)}, 101(4):473--494, 2014.

\bibitem[CP14]{CP}
Fran{\c{c}}ois Charles and Gianluca Pacienza.
\newblock Families of rational curves on holomorphic symplectic varieties and
  applications to 0-cycles.
\newblock {\em arXiv preprint arXiv:1401.4071v2}, 2014.

\bibitem[CS97]{CS}
Luca Chiantini and Edoardo Sernesi.
\newblock Nodal curves on surfaces of general type.
\newblock {\em Mathematische Annalen}, 307(1):41--56, 1997.

\bibitem[Fer12]{Fe12a}
Andrea Ferretti.
\newblock Special subvarieties of {EPW} sextics.
\newblock {\em Math. Z.}, 272(3-4):1137--1164, 2012.

\bibitem[FLVS19]{FLVS}
Lie Fu, Robert Laterveer, Charles Vial, and Mingmin Shen.
\newblock {The generalized Franchetta conjecture for some hyper-K{\"a}hler
  varieties}.
\newblock {\em Journal de Math{\'e}matiques Pures et Appliqu{\'e}es}, 2019.

\bibitem[Fog68]{Fogarty68}
John Fogarty.
\newblock Algebraic families on an algebraic surface.
\newblock {\em Amer. J. Math}, 90:511--521, 1968.

\bibitem[FP97]{FultonPandharipande95}
W.~Fulton and R.~Pandharipande.
\newblock Notes on stable maps and quantum cohomology.
\newblock In {\em Algebraic geometry---{S}anta {C}ruz 1995}, volume~62 of {\em
  Proc. Sympos. Pure Math.}, pages 45--96. Amer. Math. Soc., Providence, RI,
  1997.

\bibitem[Fu13]{Fu13a}
L.~Fu.
\newblock Decomposition of small diagonals and {C}how rings of hypersurfaces
  and {C}alabi-{Y}au complete intersections.
\newblock {\em Adv. Math.}, 244:894--924, 2013.

\bibitem[Fu15]{Fu13b}
Lie Fu.
\newblock Beauville-{V}oisin conjecture for generalized {K}ummer varieties.
\newblock {\em Int. Math. Res. Not. IMRN}, (12):3878--3898, 2015.

\bibitem[Ful98]{FultonIT}
W.~Fulton.
\newblock {\em {Intersection theory. 2nd ed.}}
\newblock {Ergebnisse der Mathematik und ihrer Grenzgebiete. 3. Folge. 2.
  Berlin: Springer}, 1998.

\bibitem[GHS10]{GritskenkoHulekSankaran10}
V.~Gritsenko, K.~Hulek, and G.~K. Sankaran.
\newblock Moduli spaces of irreducible symplectic manifolds.
\newblock {\em Compos. Math.}, 146(2):404--434, 2010.

\bibitem[GK14]{GK}
Concettina Galati and Andreas~Leopold Knutsen.
\newblock On the existence of curves with $ a\_k $-singularities on $ k3
  $-surfaces.
\newblock {\em Math. Res. Lett.}, 21:1069--1109, 2014.

\bibitem[HT10]{HassettTschinkel10}
Brendan Hassett and Yuri Tschinkel.
\newblock Intersection numbers of extremal rays on holomorphic symplectic
  varieties.
\newblock {\em Asian J. Math.}, 14(3):303--322, 2010.

\bibitem[Huy99]{Huybrechts99}
D.~Huybrechts.
\newblock Compact hyper-{K}\"ahler manifolds: basic results.
\newblock {\em Invent. Math.}, 135(1):63--113, 1999.

\bibitem[Huy14]{huybrechts2013curves}
Daniel Huybrechts.
\newblock Curves and cycles on k3 surfaces.
\newblock {\em Algebraic Geometry}, 1:69--106, 2014.

\bibitem[KLCM19]{klm2}
Andreas Knutsen, Margherita Lelli-Chiesa, and Giovanni Mongardi.
\newblock Wall divisors and algebraically coisotropic subvarieties of
  irreducible holomorphic symplectic manifolds.
\newblock {\em Transactions of the American Mathematical Society},
  371(2):1403--1438, 2019.

\bibitem[KMPS10]{KlemmMaulikPandharipandeScheidegger10}
A.~Klemm, D.~Maulik, R.~Pandharipande, and E.~Scheidegger.
\newblock Noether-{L}efschetz theory and the {Y}au-{Z}aslow conjecture.
\newblock {\em J. Amer. Math. Soc.}, 23(4):1013--1040, 2010.

\bibitem[KS19]{kreck}
Matthias Kreck and Yang Su.
\newblock Finiteness and infiniteness results for torelli groups of (hyper-)
  k$\backslash$" ahler manifolds.
\newblock {\em arXiv preprint arXiv:1907.05693}, 2019.

\bibitem[Lat18]{lat-rem}
Robert Laterveer.
\newblock {A remark on Beauville's splitting property}.
\newblock {\em Manuscripta Mathematica}, pages 1--9, 2018.

\bibitem[Lin16]{lin-kummer}
Hsueh-Yung Lin.
\newblock {On the Chow group of zero-cycles of a generalized Kummer variety}.
\newblock {\em Advances in Mathematics}, 298:448--472, 2016.

\bibitem[Lin18]{lin-lagr}
Hsueh-Yung Lin.
\newblock {Lagrangian constant cycle subvarieties in Lagrangian fibrations}.
\newblock {\em IMRN}, 02 2018.

\bibitem[Loo20]{Looijenga}
Eduard Looijenga.
\newblock Teichm{\"u}ller spaces and torelli theorems for hyperk{\"a}hler
  manifolds.
\newblock {\em Mathematische Zeitschrift}, pages 1--19, 2020.

\bibitem[LP15]{LP2}
Christian Lehn and Gianluca Pacienza.
\newblock Stability of coisotropic fibrations on holomorphic symplectic
  manifolds.
\newblock {\em arXiv preprint arXiv:1512.08672, to appear in the Annali S.N.S.
  di Pisa}, 2015.

\bibitem[Mar10]{Markman10}
Eyal Markman.
\newblock Integral constraints on the monodromy group of the hyper{K}\"ahler
  resolution of a symmetric product of a {$K3$} surface.
\newblock {\em Internat. J. Math.}, 21(2):169--223, 2010.

\bibitem[Mar11]{Markman11}
E.~Markman.
\newblock A survey of {T}orelli and monodromy results for
  holomorphic-symplectic varieties.
\newblock In {\em Proceedings of the conference "Complex and differential
  geometry"}, pages 257--322. Springer Proceedings in Mathematics, 2011.

\bibitem[MM83]{MoriMukai83}
Shigefumi Mori and Shigeru Mukai.
\newblock The uniruledness of the moduli space of curves of genus {$11$}.
\newblock In {\em Algebraic geometry ({T}okyo/{K}yoto, 1982)}, volume 1016 of
  {\em Lecture Notes in Math.}, pages 334--353. Springer, Berlin, 1983.

\bibitem[MO20]{MO}
Giovanni Mongardi and John~Christian Ottem.
\newblock Curve classes on irreducible holomorphic symplectic varieties.
\newblock {\em Communications in Contemporary Mathematics}, 22(07):1950078,
  2020.

\bibitem[MP]{MPcorr}
Giovanni Mongardi and Gianluca Pacienza.
\newblock Corrigendum and addendum to ``polarized parallel transport and
  uniruled divisors on deformations of generalized kummer varieties''.
\newblock {\em forthcoming}.

\bibitem[MP17]{MP}
Giovanni Mongardi and Gianluca Pacienza.
\newblock Polarized parallel transport and uniruled divisors on deformations of
  generalized kummer varieties.
\newblock {\em International Mathematics Research Notices},
  2018(11):3606--3620, 2017.

\bibitem[Nik79]{Nikulin80}
V.~V. Nikulin.
\newblock Integer symmetric bilinear forms and some of their geometric
  applications.
\newblock {\em Izv. Akad. Nauk SSSR Ser. Mat.}, 43(1):111--177, 238, 1979.

\bibitem[OSY18]{zal}
Georg Oberdieck, Junliang Shen, and Qizheng Yin.
\newblock Rational curves in the fano varieties of cubic 4-folds and
  gromov-witten invariants.
\newblock {\em arXiv preprint arXiv:1805.07001}, 2018.

\bibitem[Ran95]{Ran95}
Z.~Ran.
\newblock Hodge theory and deformation of maps.
\newblock {\em Compositio Mathematica}, 97(3):309--328, 1995.

\bibitem[Rie]{Riess14}
U.~Riess.
\newblock { On the Beauville conjecture}.
\newblock {\em IMRN doi: 10.1093/imrn/rnv315}.

\bibitem[Ser]{sernesi2007deformations}
Edoardo Sernesi.
\newblock {\em Deformations of algebraic schemes}, volume 334 of {\em
  Grundlehren der mathematischen Wissenschaften}.

\bibitem[SV16]{SV}
Mingmin Shen and Charles Vial.
\newblock {\em The Fourier transform for certain hyperK{\"a}hler fourfolds},
  volume 240.
\newblock American Mathematical Society, 2016.

\bibitem[SY17]{shen2017k3}
Junliang Shen and Qizheng Yin.
\newblock $ k3 $ categories, one-cycles on cubic fourfolds, and the
  beauville--voisin filtration.
\newblock {\em Journal of the Institute of Mathematics of Jussieu}, pages
  1--27, 2017.

\bibitem[SYZ17]{shen2017derived}
Junliang Shen, Qizheng Yin, and Xiaolei Zhao.
\newblock Derived categories of k3 surfaces, o'grady's filtration, and
  zero-cycles on holomorphic symplectic varieties.
\newblock {\em arXiv preprint arXiv:1705.06953}, 2017.

\bibitem[Ver13]{Verbitsky09}
Misha Verbitsky.
\newblock Mapping class group and a global {T}orelli theorem for hyperk\"ahler
  manifolds.
\newblock {\em Duke Math. J.}, 162(15):2929--2986, 2013.
\newblock Appendix A by Eyal Markman.

\bibitem[Ver15]{Verbitsky13}
Misha Verbitsky.
\newblock Ergodic complex structures on hyperk{\"a}hler manifolds.
\newblock {\em Acta Mathematica}, 215(1):161--182, 2015.

\bibitem[Ver17]{verbitsky2017ergodic}
Misha Verbitsky.
\newblock Ergodic complex structures on hyperkahler manifolds: an erratum.
\newblock {\em arXiv preprint arXiv:1708.05802}, 2017.

\bibitem[Ver19]{Verbitsky-erratum}
Misha Verbitsky.
\newblock Mapping class group and global torelli theorem for hyperkahler
  manifolds: an erratum.
\newblock {\em arXiv preprint arXiv:1908.11772}, 2019.

\bibitem[Via17]{vial2017motive}
Charles Vial.
\newblock On the motive of some hyperk{\"a}hler varieties.
\newblock {\em Journal f{\"u}r die reine und angewandte Mathematik (Crelles
  Journal)}, 2017(725):235--247, 2017.

\bibitem[Voi08]{Voisin08}
Claire Voisin.
\newblock On the {C}how ring of certain algebraic hyper-{K}\"ahler manifolds.
\newblock {\em Pure Appl. Math. Q.}, 4(3, Special Issue: In honor of Fedor
  Bogomolov. Part 2):613--649, 2008.

\bibitem[Voi15]{voisin2015rational}
Claire Voisin.
\newblock Rational equivalence of 0-cycles on k3 surfaces and conjectures of
  huybrechts and o'grady.
\newblock In {\em Recent advances in algebraic geometry}, volume 417, pages
  422--436. Cambridge University Press Cambridge, 2015.

\bibitem[Voi16]{Voi15}
Claire Voisin.
\newblock Remarks and questions on coisotropic subvarieties and 0-cycles of
  hyper-k{\"a}hler varieties.
\newblock In {\em K3 surfaces and their moduli}, pages 365--399. Springer,
  2016.

\bibitem[Voi18]{voisin2018triangle}
Claire Voisin.
\newblock {Triangle varieties and surface decomposition of hyper-K\"ahler
  manifolds}.
\newblock {\em arXiv preprint arXiv:1810.11848}, 2018.

\bibitem[Yin15]{yinfinite}
Qizheng Yin.
\newblock {Finite-dimensionality and cycles on powers of K3 surfaces}.
\newblock {\em Comment. Math. Helv}, 90:503--511, 2015.

\bibitem[YZ96]{YauZaslow95}
Shing-Tung Yau and Eric Zaslow.
\newblock B{PS} states, string duality, and nodal curves on {$K3$}.
\newblock {\em Nuclear Phys. B}, 471(3):503--512, 1996.

\end{thebibliography}

\end{document}